\documentclass[11pt,a4paper]{article} 

\usepackage{graphicx,latexsym,euscript,makeidx,color,bm}
\usepackage{amsmath,amsfonts,amssymb,amsthm,thmtools,mathtools,mathrsfs,enumerate}
\usepackage[colorlinks,linkcolor=blue,anchorcolor=green,citecolor=red]{hyperref}

\usepackage{tikz}
\usetikzlibrary{fit,calc,positioning} 
\tikzstyle{Block}=[rectangle,minimum width=3cm,minimum height=1cm,text centered,text width=5.2cm,draw=black]
\tikzstyle{Implication}=[rectangle,minimum width=3cm,minimum height=1cm,text centered,text width=5.2cm]
\tikzstyle{jian}=[<->, >=stealth]

\usepackage{geometry}
\allowdisplaybreaks[4]

  \def\cA{{\cal A}}  
  \def\cB{{\cal B}}  
  \def\cC{{\cal C}}  
  \def\cD{{\cal D}}  
\def\dbE{\mathbb{E}}    
\def\dbF{\mathbb{F}}  \def\cF{{\cal F}}  
  \def\cG{{\cal G}}  
\def\dbH{\mathbb{H}}    
    
  \def\cJ{{\cal J}}  
    
  \def\cL{{\cal L}}  
  \def\cM{{\cal M}}  
  \def\cN{{\cal N}}  
    
\def\dbP{\mathbb{P}}  \def\cP{{\cal P}}  
  \def\cQ{{\cal Q}}  
\def\dbR{\mathbb{R}}  \def\cR{{\cal R}}  
\def\dbS{\mathbb{S}}  \def\cS{{\cal S}}  
    
  \def\cU{{\cal U}}  
  \def\cV{{\cal V}}  
    
  \def\cX{{\cal X}}  \def\Bx{{\bm x}}

\def\ss{\smallskip}      \def\lt{\left}       \def\hb{\hbox}
\def\ms{\medskip}        \def\rt{\right}      \def\ae{\hb{a.e.}}
\def\bs{\bigskip}        \def\lan{\langle}    \def\as{\hb{a.s.}}
\def\ds{\displaystyle}   \def\ran{\rangle}    \def\tr{\hb{tr$\,$}}
\def\ts{\textstyle}         
\def\no{\noindent}          
\def\hp{\hphantom}         
\def\nn{\nonumber}         
\def\rf{\eqref}          \def\Blan{\Big\lan}  
\def\cd{\cdot}           \def\Bran{\Big\ran}  
\def\deq{\triangleq}     \def\({\Big (}       \def\ba{\begin{aligned}}
\def\les{\leqslant}      \def\){\Big )}       \def\ea{\end{aligned}}
\def\ges{\geqslant}      \def\[{\Big[}        \def\bel{\begin{equation}\label}
          \def\]{\Big]}        \def\ee{\end{equation}}
      \def\q{\quad}        
\def\h{\widehat}               
                                              
\def\a{\alpha}        \def\Om{\Omega}  
\def\b{\beta}      \def\d{\delta}   \def\F{\Phi}     
   \def\Th{\Theta}  \def\th{\theta}  \def\Si{\Sigma}  
\def\f{\varphi} \def\L{\Lambda}  \def\l{\lambda}        \def\e{\varepsilon}
\def\t{\tau}    \def\i{\infty}      

\newtheoremstyle{thry}
{}      
{}      
{\sl}   
{}      
{\bf}   
{.}     
{.5em}  
{}      

\theoremstyle{thry}

\newtheorem{theorem}{Theorem}[section]
\newtheorem{proposition}[theorem]{Proposition}
\newtheorem{corollary}[theorem]{Corollary}
\newtheorem{lemma}[theorem]{Lemma}

\theoremstyle{definition}
\newtheorem{definition}[theorem]{Definition}
\newtheorem{example}[theorem]{Example}

\newtheorem{taggedthmx}{}
\newenvironment{taggedthm}[1]
 {\renewcommand\thetaggedthmx{#1}\taggedthmx}
 {\endtaggedthmx}

\theoremstyle{remark}

\makeatletter
   
   \@addtoreset{equation}{section}
   \newcommand{\setword}[2]{%
   \phantomsection
   #1\def\@currentlabel{\unexpanded{#1}}\label{#2}%
   }
\makeatother

\begin{document}

\title{\bf Two-Person Zero-Sum Stochastic Linear-Quadratic Differential Games}
\author{Jingrui Sun\thanks{Department of Mathematics, Southern University of Science and Technology,
                   Shenzhen, Guangdong, 518055, China (Email: {\tt sunjr@sustech.edu.cn}).
                   This work is supported by NSFC Grant 11901280.} }

\maketitle

\no{\bf Abstract.}
The paper studies the open-loop saddle point and the open-loop lower and upper values,
as well as their relationship for two-person zero-sum stochastic linear-quadratic
(LQ, for short) differential games with deterministic coefficients.
It derives a necessary condition for the finiteness of the open-loop lower and upper values
and a sufficient condition for the existence of an open-loop saddle point.
It turns out that under the sufficient condition, a strongly regular solution to the associated
Riccati equation uniquely exists, in terms of which a closed-loop representation is further
established for the open-loop saddle point.
Examples are presented to show that the finiteness of the open-loop lower and upper values
does not ensure the existence of an open-loop saddle point in general.
But for the classical deterministic LQ game, these two issues are equivalent and both imply
the solvability of the Riccati equation, for which an explicit representation of the solution
is obtained.

\ms
\no{\bf Key words.}
linear-quadratic differential game, two-person, zero-sum, open-loop, lower value, upper value,
saddle point, Riccati equation, closed-loop representation.

\ms
\no{\bf AMS subject classifications.} 93E20, 91A23, 49N70.

\section{Introduction}\label{Sec:Intro}

Let $(\Om,\cF,\dbP)$ be a complete probability space on which a standard one-dimensional
Brownian motion $W=\{W(t);t\ges0\}$ is defined, and let $\dbF=\{\cF_t\}_{t\ges0}$ be
the usual augmentation of the natural filtration generated by $W$.
Consider the following controlled linear stochastic differential equation (SDE, for short)
on a finite horizon $[0,T]$:
\begin{equation}\label{state}
\left\{\begin{aligned}
dX(t) &= \big[A(t)X(t)+B_1(t)u_1(t)+B_2(t)u_2(t)\big]dt \\
      &\hp{=\ } +\big[C(t)X(t)+D_1(t)u_1(t)+D_2(t)u_2(t)\big]dW(t), \\
 X(0) &= x,
\end{aligned}\right.
\end{equation}
where $A,C:[0,T]\to\dbR^{n\times n}$ and $B_i,D_i:[0,T]\to\dbR^{n\times m_i}$ $(i=1,2)$,
called the {\it coefficients} of the {\it state equation} \rf{state}, are given bounded
deterministic functions; the process $u_i$ ($i=1,2$), belonging to the space
\begin{align*}
\cU_i[0,T] &=\ts\Big\{\f:[0,T]\times\Om\to\dbR^{m_i} \bigm|
             \f~\hb{is $\dbF$-progressivley measurable,}~\dbE\int_0^T|\f(t)|^2dt<\i \Big\},
\end{align*}
is the control of Player $i$; and $x\in\dbR^n$ is a given {\it initial state}.
The criterion for the performance of $u_1$ and $u_2$ is given by the following quadratic functional:
\begin{align}\label{Q-function}
J(x;u_1,u_2) = \dbE\Bigg\{\lan GX(T),X(T)\ran +\int_0^T\Blan\!
   \begin{pmatrix}Q   & \!S_1^\top & \!S_2^\top \\
                  S_1 & \!R_{11}   & \!R_{12}   \\
                  S_2 & \!R_{21}   & \!R_{22}   \end{pmatrix}\!
   \begin{pmatrix}X \\ u_1 \\ u_2  \end{pmatrix}\!,
   \begin{pmatrix}X \\ u_1 \\ u_2  \end{pmatrix}\! \Bran dt\Bigg\},
\end{align}
where $G$ is an $n\times n$ symmetric real matrix; $Q:[0,T]\to\dbR^{n\times n}$, $S_i:[0,T]\to\dbR^{m_i\times n}$,
and $R_{ij}:[0,T]\to\dbR^{m_i\times m_j}$ ($i,j=1,2$) are bounded functions with
$$ Q(t)^\top = Q(t),
\q \begin{pmatrix}R_{11}(t) & \!R_{12}(t) \\
                  R_{21}(t) & \!R_{22}(t) \end{pmatrix}^\top
 = \begin{pmatrix}R_{11}(t) & \!R_{12}(t) \\
                  R_{21}(t) & \!R_{22}(t) \end{pmatrix}, \q\forall t\in[0,T]. $$
In the Lebesgue integral on the right-hand side of \rf{Q-function}, the variable $t$ is suppressed for convenience.

\ms

The functional \rf{Q-function} can be regarded as the loss of Player 1 and the gain of Player 2.
So in this {\it two-person zero-sum stochastic linear-quadratic differential game} (Problem (SLQG), for short),
Player 1 wants to find his/her control that minimizes the loss, while Player 2 wants to find
his/her control that maximizes the gain.
The best choice for the two players is a control pair such that no one can benefit
by changing his/her control while the other keeps his/her unchanged.
Such a pair $(u_1^*,u_2^*)$ is called an {\it open-loop saddle point}, mathematically
defined by the following inequalities:
\begin{align*}
J(x;u_1^*,u_2) \les J(x;u_1^*,u_2^*) \les J(x;u_1,u_2^*), \q\forall (u_1,u_2)\in\cU_1[0,T]\times\cU_2[0,T].
\end{align*}
Another two important notions in game theory are the {\it open-loop lower and upper values} defined as
\begin{align*}
 V^-(x) = \sup_{u_2\in\cU_2[0,T]} \inf_{u_1\in\cU_1[0,T]} J(x;u_1,u_2) \q\hb{and}\q
 V^+(x) = \inf_{u_1\in\cU_1[0,T]} \sup_{u_2\in\cU_2[0,T]} J(x;u_1,u_2),
\end{align*}
respectively. It is clear that $V^-(x)\les V^+(x)$. In the case of $V^-(x)=V^+(x)$,
we denote by $V(x)$ the common value and say that the game has an {\it open-loop value} at $x$.

\ms

Linear-quadratic (LQ, for short) differential games constitute an important class of differential games.
They are widely encountered in many fields, such as engineering, economy, and biology,
and also play an essential role in the study of general differential games
(see, for example, \cite{Engwerda2009,Carmona2016}).
The study of deterministic LQ differential games (Problem (DLQG), for short),
in which the state evolves according to an ordinary differential equation (ODE, for short),
can be traced back to the work of Ho--Bryson--Baron \cite{Ho-Bryson-Baron1965},
in the context of a linearized pursuit-evasion game.
Later, Schmitendorf \cite{Schmitendorf1970} studied the open-loop and closed-loop
strategies for Problem (DLQG) in a rigorous framework and showed that the existence
of a closed-loop saddle point might not imply the existence of an open-loop saddle point.
In 1979, Bernhard \cite{Bernhard1979} considered the zero-sum game with the additional
restriction on the final state from a closed-loop point of view; see also the follow-up
work of Ba\c{s}ar--Bernhard \cite{Basar-Bernhard1995}.
In 2005, Zhang \cite{Zhang2005} established the equivalence among the existence of a finite
open-loop value, the finiteness of open-loop lower and upper values, and the existence
of an open-loop saddle point for a class of deterministic LQ differential games.
The results of Zhang \cite{Zhang2005} were later sharpened by Delfour \cite{Delfour2007}
in 2007 and generalized to closed-loop LQ differential games by Delfour--Sbarba
\cite{Delfour-Sbarba2009} in 2009.
Two-person zero-sum stochastic LQ differential games (Problem (SLQG)) have also been considered
by many authors.
Mou--Yong \cite{Mou-Yong2006} studied Problem (SLQG) from an open-loop point of view by means
of the Hilbert space method.
Sun--Yong \cite{Sun-Yong2014} established the characterizations of open-loop and closed-loop
saddle points for Problem (SLQG) and investigated their relationship
(see also the books \cite{Yong2015,Sun-Yong2020b}).
Based on the idea in \cite{Sun-Yong2014}, Sun--Yong--Zhang \cite{Sun-Yong-Zhang2016} further
explored stochastic LQ differential games over infinite horizons.
Yu \cite{Yu2015} studied the optimal feedback control-strategy pair for Problem (SLQG) using
a Riccati equation approach.
There are many other works on LQ differential games, among which we would like to mention
the works \cite{Hamadene1998,Hamadene1999,Sun-Yong2019} on nonzero-sum LQ games and
the works \cite{Bensoussan-Sung-Yam-Yung2016,Ran-Yu-Zhang2020} on mean-field LQ games.

\ms

In this paper, the analysis of the above two-person zero-sum stochastic LQ differential game
mainly focuses on the open-loop saddle point and the open-loop lower and upper values, as well
as their relationship.
Our approach is partially based on the recently developed results on two-person zero-sum
stochastic LQ differential games and indefinite stochastic LQ optimal control problems (see \cite{Sun-Yong2014,Sun-Li-Yong2016}).
The main contribution of this paper can be briefly summarized as follows
(A complete summary of the results is presented in the conclusion section; see \autoref{fig:summary}).

\ms

(i) It is found that in general the finiteness of the open-loop lower and upper values does not
ensure the existence of an open-loop saddle point (see \autoref{exm:saddle-novuale}), which is
different from Zhang's equivalence result \cite{Zhang2005} for deterministic LQ differential games.
In fact, the finiteness of the open-loop lower and upper values does not even imply the existence
of an open-loop value (see \autoref{V-<V<V+})
The reason is that in \cite{Zhang2005}, the stochastic part is absent and an additional assumption
is imposed on the weighting matrices for the controls, i.e., $R_{11}$ is required to be uniformly
positive definite and $R_{22}$ is required to be uniformly negative definite.

\ms

(ii) A necessary condition and a sufficient condition are derived for the existence of finite
open-loop values. These two conditions are closely related to indefinite stochastic LQ optimal
problems. It is shown that under the sufficient condition, the associated Riccati equation is
{\it strongly regularly solvable} (see \autoref{thm:main-Ric}), and consequently, a unique
open-loop saddle point exists for every initial state and admits a closed-loop representation
(see \autoref{thm:main-saddle}).
The solvability of the Riccati equation constitutes the most difficult part of the paper.
We overcome this difficulty by exploring the connection between the stochastic LQ differential
game and two stochastic LQ optimal control problems and examining the local existence of solutions
for the Riccati equation.

\ms

(iii) For the deterministic two-person zero-sum LQ differential game, which can be regarded as
a special case of the stochastic game, we establish the equivalence between the existence of an
open-loop saddle point and the finiteness of the open-loop lower and upper values by a new approach
(see \autoref{thm:DLQG-main}).
More importantly, we find that in the deterministic case, the finiteness of the open-loop lower and
upper values also implies the solvability of the Riccati equation, for which we obtain an explicit
representation of the solution (see \autoref{thm:DLQG-main} and \autoref{crllry:rep-P}).

\ms

The rest of the paper is organized as follows.
In \autoref{Sec:Preliminary} we give the preliminaries and collect some recently developed results
on stochastic LQ optimal control problems.
In \autoref{Sec:Value} we study the open-loop lower and upper values,
and in \autoref{Sec:Saddle} we establish the solvability of the associated Riccati equation as well
as the closed-loop representation of the open-loop saddle point.
In \autoref{Sec:Relation}, the relationship between the open-loop saddle and the open-loop lower and
upper values is discussed, and an equivalence result is presented for deterministic two-person zero-sum
LQ differential games. \autoref{Sec:Conclusion} concludes the paper.

\section{Preliminaries}\label{Sec:Preliminary}

Throughout this paper, $\dbR^{n\times m}$ denotes the Euclidean space of $n\times m$ real
matrices, equipped with the Frobenius inner product
$$ \lan M,N\ran=\tr(M^\top N), \q M,N\in\dbR^{n\times m}, $$
where $M^\top$ is the transpose of $M$ and $\tr(M^\top N)$ is the trace of $M^\top N$.
The norm induced by the Frobenius inner product is denoted by $|\cd|$.
The identity matrix of size $n$ is denoted by $I_n$, which is often simply written as $I$
when no confusion occurs.
Let $\dbS^n$ be the space of symmetric $n\times n$ real matrices and $\bar\dbS^n_+$ the
space of symmetric positive semidefinite $n\times n$ real matrices.
For $\dbS^n$-valued functions $M,N$ on $[0,T]$, we write $M\ges N$ (respectively, $M\les N$)
if $M(t)-N(t)\in\bar\dbS^n_+$ (respectively, $N(t)-M(t)\in\bar\dbS^n_+$)
for almost every $t\in[0,T]$, and we write $M\gg N$ (respectively, $M\ll N$)
if there exists a constant $\a>0$ such that $M(t)-N(t)\ges\a I$ (respectively, $N(t)-M(t)\ges\a I$)
for almost every $t\in[0,T]$.

\ms

Recall that $\dbF=\{\cF_t\}_{t\ges0}$ is the usual augmentation of the natural filtration
generated by the Brownian motion $W$.
For a process $\f$, we write $\f\in\dbF$ if it is $\dbF$-progressively measurable.
Let $\dbH$ be a subset of some Euclidean space. In the following table we list some spaces
that will be frequently used in the sequel.
\begin{align*}
C([0,T];\dbH)
  &\ts= \Big\{\f:[0,T]\to\dbH\bigm|\f~\hb{is continuous} \Big\}, \\
L^\i(0,T;\dbH)
  &\ts= \Big\{\f:[0,T]\to\dbH\bigm|\f~\hb{is Lebesgue essentially bounded} \Big\}, \\
L^2(0,T;\dbH)
  &\ts= \Big\{\f:[0,T]\to\dbH\bigm|\int_0^T|\f(t)|^2dt<\i \Big\}, \\
L_\dbF^2(0,T;\dbH)
  &\ts= \Big\{\f:[0,T]\times\Om\to\dbH\bigm|\f\in\dbF~\hb{and}~\dbE\int_0^T|\f(t)|^2dt<\i \Big\}.
\end{align*}
Note that $L_\dbF^2(0,T;\dbR^m)$ is a Hilbert space under the usual product
$$ [\![ u,v ]\!] = \dbE\int_0^T \lan u(t),v(t)\ran dt, \q u,v\in L_\dbF^2(0,T;\dbR^m). $$
We denote the induced norm of a process $u\in L_\dbF^2(0,T;\dbR^m)$ by $\|u\|$.
In terms of the above notation, we see that for Player $i$, the space of controls is
$$ \cU_i[0,T] = L_\dbF^2(0,T;\dbR^{m_i}). $$
As mentioned in the introduction section, we assume that the coefficients of the state equation \rf{state}
and the weighting matrices in the quadratic functional \rf{Q-function} satisfy the following conditions.
\begin{enumerate}\setlength{\itemindent}{3.7pt}
\item[{\bf\setword{(A1)}{A1}}]
     $A,C:[0,T]\to\dbR^{n\times n}$ and $B_i,D_i:[0,T]\to\dbR^{n\times m_i}$ $(i=1,2)$ are bounded,
     Lebesgue measurable functions, i.e.,
     $$ A,C\in L^\i(0,T;\dbR^{n\times n}), \q B_i,D_i\in L^\i(0,T;\dbR^{n\times m_i}). $$
\item[{\bf\setword{(A2)}{A2}}]
     $G\in\dbS^n$, $Q\in L^\i(0,T;\dbS^n)$, and for $i,j=1,2$,
     $$ S_i\in L^\i(0,T;\dbR^{m_i\times n}),
        \q R_{ij}\in L^\i(0,T;\dbR^{m_i\times m_j}), \q R_{ij}^\top =R_{ji}. $$
\end{enumerate}

\ss

Next we collect some results from stochastic LQ optimal control theory.
Consider the state equation
\begin{equation}\label{state:SLQ}\left\{\begin{aligned}
d\cX(t) &= \big[\cA(t)\cX(t)+\cB(t)v(t)\big]dt + \big[\cC(t)\cX(t)+\cD(t)v(t)\big]dW(t),\q t\in[0,T], \\
 \cX(0) &= x,
\end{aligned}\right.\end{equation}
and the cost functional
\begin{align}\label{cost:SLQ}
\cJ(x;v) =\dbE\Bigg\{\lan\cG\cX(T),\cX(T)\ran
+\int_0^T\Blan\!\begin{pmatrix}\cQ(t) & \!\cS(t)^\top \\
                               \cS(t) & \!\cR(t)      \end{pmatrix}\!
                \begin{pmatrix}\cX(t) \\ v(t) \end{pmatrix}\!,
                \begin{pmatrix}\cX(t) \\ v(t) \end{pmatrix}\! \Bran dt \Bigg\},
\end{align}
where in \rf{state:SLQ},
$$ \cA,\cC \in L^\i(0,T;\dbR^{n\times n}), \q \cB,\cD \in L^\i(0,T;\dbR^{n\times m}), $$
and in \rf{cost:SLQ},
$$ \cG\in\dbS^n, \q \cQ\in L^\i(0,T;\dbS^n), \q \cS\in L^\i(0,T;\dbR^{m\times n}), \q\cR\in L^\i(0,T;\dbS^m). $$
The stochastic LQ optimal control problem is as follows.

\begin{taggedthm}{Problem (SLQ)}\label{Prob:SLQ}
For a given initial state $x\in\dbR^n$, find a control $v^*\in L_\dbF^2(0,T;\dbR^m)$ such that
\begin{equation}\label{SLQ:v*}
  \cJ(x;v^*) = \inf_{v\in L_\dbF^2(0,T;\dbR^m)}\cJ(x;v) \equiv \cV(x).
\end{equation}
\end{taggedthm}

The control $v^*\in L_\dbF^2(0,T;\dbR^m)$ in \rf{SLQ:v*} is called an {\it open-loop optimal control}
for the initial state $x$, and $\cV(x)$ is called the {\it value} of Problem (SLQ) at $x$.

\ms

The following lemmas summarize a few results for Problem (SLQ) that will be needed in the subsequent sections.
The reader is referred to Sun--Li--Yong \cite{Sun-Li-Yong2016} for proofs; see also the recent book \cite{Sun-Yong2020}
by Sun--Yong.

\begin{lemma}\label{lmm:SLQ-main-1}
If $\cV(x)>-\i$ for some initial state $x$, then
$$ \cJ(0;v)\ges0, \q\forall v\in L_\dbF^2(0,T;\dbR^m). $$
\end{lemma}

\begin{lemma}\label{lmm:SLQ-main-2}
If for some constant $\a>0$,
$$ \cJ(0;v)\ges \a\|v\|^2, \q\forall v\in L_\dbF^2(0,T;\dbR^m), $$
then the following hold:
\begin{enumerate}[\rm(i)]\setlength{\itemsep}{-1pt}
\item For every initial state $x$, a unique open-loop optimal control exists.
\item The Riccati differential equation
      \begin{equation}\label{Ric:SLQ}\left\{\begin{aligned}
      & \dot\cP + \cP\cA + \cA^\top\cP + \cC^\top\cP\cC + \cQ\\
      & \hp{\dot\cP} -(\cP\cB + \cC^\top\cP\cD + \cS^\top)(\cR+\cD^\top\cP\cD)^{-1}
                      (\cB^\top\cP + \cD^\top\cP\cC + \cS)=0, \\
      & \cP(T) = \cG
      \end{aligned}\right.\end{equation}
      admits a unique solution $\cP\in C([0,T];\dbS^n)$ such that
      \begin{equation*}
       \cR+\cD^\top\cP\cD \gg 0.
      \end{equation*}
      In particular, if $\cG\ges0$, $\cQ\ges0$, and $\cR\gg0$,
      then \rf{Ric:SLQ} has a unique nonnegative solution $\cP\in C([0,T];\bar\dbS^n_+)$.
\item The unique open-loop optimal control $v^*$ for the initial state $x$ admits the following
      closed-loop representation:
      $$ v^* = -(\cR+\cD^\top\cP\cD)^{-1}(\cB^\top\cP + \cD^\top\cP\cC + \cS)X, $$
      and the value at $x$ is given by
      $$ \cV(x) = \lan \cP(0)x,x\ran. $$
\end{enumerate}
\end{lemma}


\section{Open-loop lower and upper values}\label{Sec:Value}

In this section we study the open-loop lower and upper values of the two-person zero-sum stochastic
LQ differential game. We derive a necessary condition and a sufficient condition for the finiteness
of the open-loop lower and upper values. First, let us recall the following definition.

\begin{definition}
The {\it open-loop lower value} $V^-(x)$ and the {\it open-loop upper value} $V^+(x)$
at the initial state $x\in\dbR^n$ are defined by
\begin{align*}
V^-(x) = \sup_{u_2\in\cU_2[0,T]} \inf_{u_1\in\cU_1[0,T]} J(x;u_1,u_2) \q\hb{and}\q
V^+(x) = \inf_{u_1\in\cU_1[0,T]} \sup_{u_2\in\cU_2[0,T]} J(x;u_1,u_2),
\end{align*}
respectively. Note that for every $x\in\dbR^n$,
$$ V^-(x) \les V^+(x). $$
If the above holds with equality, we call the common value, denoted by $V(x)$,
an {\it open-loop value} at the initial state $x$.
\end{definition}

It is shown in \cite{Zhang2005} that for a special class of deterministic two-person zero-sum LQ
differential games, if both the open-loop lower and upper values are finite, they must be equal.
However, this result does not hold in general. Here is an example.

\begin{example}\label{V-<V<V+}
Consider the one-dimensional state equation
$$\left\{\begin{aligned}
dX(t) &= \sqrt{t}\,u_1(t)dt + t\,u_2(t) dW(t),\q t\in[0,1], \\
 X(0) &= x,
\end{aligned}\right.$$
and the quadratic functional
$$ J(x;u_1,u_2) = \dbE\bigg\{ |X(1)|^2 + \int_0^1 \[2tu_1(t)u_2(t)-t^2|u_2(t)|^2\] dt\bigg\}. $$
We claim that
$$ V^+(x)= x^2, \q V^-(x)=0, \q\forall x\in\dbR. $$
To verify the claim, we observe first that
\begin{align*}
\dbE|X(1)|^2
&= \dbE\bigg[x+\int_0^1\sqrt{t}\,u_1(t)dt\bigg]^2 + \dbE\int_0^1 t^2|u_2(t)|^2dt \\
&~\hp{=} +2\dbE\bigg[\int_0^1\sqrt{t}\,u_1(t)dt\int_0^1t\,u_2(t)dW(t)\bigg].
\end{align*}
It follows that
\begin{align}\label{ex3.2:J=}
J(x;u_1,u_2)
&= \dbE\bigg[x+\int_0^1\sqrt{t}\,u_1(t)dt\bigg]^2 + 2\dbE\int_0^1 tu_1(t)u_2(t)dt \nn\\
&~\hp{=} +2\dbE\bigg[\int_0^1\sqrt{t}\,u_1(t)dt\int_0^1t\,u_2(t)dW(t)\bigg].
\end{align}
Clearly, we have
\begin{align}\label{ex3.2:u1=0}
 \sup_{u_2\in\cU_2[0,1]} J(x;0,u_2) = x^2.
\end{align}
For $u_1\ne0$ and $u_2=\l u_1$ ($\l>0$),
\begin{align*}
2\dbE\bigg[\int_0^1\sqrt{t}\,u_1(t)dt\int_0^1t\,u_2(t)dW(t)\bigg]
&= 2\l\dbE\bigg[\int_0^1\sqrt{t}\,u_1(t)dt\int_0^1t\,u_1(t)dW(t)\bigg]\\
&\ges -\l\dbE\bigg[\bigg(\int_0^1\sqrt{t}\,u_1(t)dt\bigg)^2 + \bigg(\int_0^1t\,u_1(t)dW(t)\bigg)^2 \bigg] \\
&\ges -\l\dbE\int_0^1t|u_1(t)|^2dt - \l\dbE\int_0^1t^2|u_1(t)|^2dt.
\end{align*}
Thus, we have
\begin{align*}
J(x;u_1,\l u_1) \ges \dbE\bigg[x+\int_0^1\sqrt{t}\,u_1(t)dt\bigg]^2 + \l\dbE\int_0^1 (t-t^2)|u_1(t)|^2dt.
\end{align*}
Since for $u_1\ne0$,
$$ \dbE\int_0^1 (t-t^2)|u_1(t)|^2dt>0, $$
it follows that
\begin{align}\label{ex3.2:u1not=0}
\sup_{u_2\in\cU_2[0,1]} J(x;u_1,u_2)\ges \sup_{\l>0} J(x;u_1,\l u_1) =\i, \q\forall u_1\ne0.
\end{align}
Combining \rf{ex3.2:u1=0} and \rf{ex3.2:u1not=0} we obtain
$$ V^+(x)=\inf_{u_1\in\cU_1[0,1]} \sup_{u_2\in\cU_2[0,1]} J(x;u_1,u_2)=x^2. $$
To show $V^-(x)=0$, we note that for any $u_2\in\cU_2[0,1]$, $\sqrt{t}$ and $\dbE[tu_2(t)]$ are
elements of the Hilbert space $L^2(0,1;\dbR)$.
Since $\sqrt{t}$ is not in the (one-dimensional) space generated by $\dbE[tu_2(t)]$,
by the Hahn-Banach theorem there exists a $\bar u_1\in L^2(0,1;\dbR)\subseteq\cU_1[0,1]$ such that
$$ \int_0^1 \sqrt{t}\,\bar u_1(t) dt =-x \q\hb{and}\q \int_0^1 \bar u_1(t)\dbE[tu_2(t)] dt=0. $$
Together with \rf{ex3.2:J=}, it gives (noting that $\bar u_1$ is deterministic)
\begin{align*}
\inf_{u_1\in\cU_1[0,1]}J(x;u_1,u_2)\les J(x;\bar u_1,u_2)=0, \q\forall u_2\in\cU_2[0,1].
\end{align*}
On the other hand, it is trivial that
$$\inf_{u_1\in\cU_1[0,1]} J(x;u_1,0)\ges0.$$
Thus,
$$V^-(x) = \sup_{u_2\in\cU_2[0,1]} \inf_{u_1\in\cU_1[0,1]} J(x;u_1,u_2) =0.$$
This proves our claim.
\end{example}

The following result gives necessary conditions for the finiteness of the open-loop lower and upper values.

\begin{theorem}\label{thm:tu-ao}
Let {\rm\ref{A1}--\ref{A2}} hold.
\begin{enumerate}[\rm(i)]\setlength{\itemsep}{-1pt}
\item If $V^-(x)$ is finite for some initial state $x$, then
      \begin{align}\label{J1>0}
      & J(0;u_1,0) \ges0, \q\forall u_1\in\cU_1[0,T].
      \end{align}
\item If $V^+(x)$ is finite for some initial state $x$, then
      \begin{align}\label{J2<0}
      & J(0;0,u_2) \les0, \q\forall u_2\in\cU_2[0,T].
      \end{align}
\end{enumerate}
\end{theorem}

\begin{proof}
To emphasize the dependence on the initial state and the controls of the two players,
we denote the solution of the state equation \rf{state} by $X_x^{u_1,u_2}$.
We now prove (i) by contradiction.
Suppose that $J(0;\bar u_1,0)<0$ for some $\bar u_1\in\cU_1[0,T]$. Then for any $u_2\in\cU_2[0,T]$,
$$ \inf_{u_1\in\cU_1[0,T]}J(x;u_1,u_2)\les \inf_{\l\in\dbR}J(x;\l\bar u_1,u_2). $$
Substituting the relation $X_x^{\l\bar u_1,u_2} = \l X_0^{\bar u_1,0} + X_x^{0,u_2}$ into
the expression of $J(x;\l\bar u_1,u_2)$, we obtain
\begin{align*}
J(x;\l\bar u_1,u_2) = \l^2J(0;\bar u_1,0) + J(x;0,u_2) + 2\l\rho(x;\bar u_1,u_2),
\end{align*}
where
\begin{align*}
\rho(x;\bar u_1,u_2) = \dbE\Bigg\{\lan GX_0^{\bar u_1,0},X_x^{0,u_2}\ran +\int_0^T\Blan\!
   \begin{pmatrix}Q   & \!S_1^\top & \!S_2^\top \\
                  S_1 & \!R_{11}   & \!R_{12}   \\
                  S_2 & \!R_{21}   & \!R_{22}   \end{pmatrix}\!
   \begin{pmatrix}X_0^{\bar u_1,0} \\ \bar u_1 \\ 0  \end{pmatrix}\!,
   \begin{pmatrix}X_x^{0,u_2} \\ 0 \\ u_2  \end{pmatrix}\! \Bran dt\Bigg\}.
\end{align*}
Since $J(0;\bar u_1,0)<0$, it follows that
$$ \inf_{u_1\in\cU_1[0,T]}J(x;u_1,u_2)\les \inf_{\l\in\dbR}J(x;\l\bar u_1,u_2) = -\i. $$
Because in the above $u_2$ is arbitrary, we obtain the contradiction
$$ V^-(x) = \sup_{u_2\in\cU_2[0,T]} \inf_{u_1\in\cU_1[0,T]} J(x;u_1,u_2) =-\i. $$
In a similar manner we can prove (ii).
\end{proof}

\autoref{thm:tu-ao} tells us that in order for both the open-loop lower value and the open-loop
upper value to be finite, the conditions \rf{J1>0} and \rf{J2<0} must hold.
We now present an example showing that \rf{J1>0} and \rf{J2<0} do not necessarily imply the
finiteness of the open-loop lower and upper values.

\begin{example}
Consider the one-dimensional state equation
$$\left\{\begin{aligned}
dX(t) &= u_1(t)dt + u_2(t) dW(t),\q t\in[0,1], \\
 X(0) &= x,
\end{aligned}\right.$$
and the quadratic functional
$$ J(x;u_1,u_2) = \dbE\bigg\{ -|X(1)|^2 + \int_0^1 \[|u_1(t)|^2-|u_2(t)|^2\] dt\bigg\}. $$
When $x=0$ and $u_2=0$,
$$ \dbE|X(1)|^2 = \dbE\lt[\int_0^1 u_1(t) dt\rt]^2 \les \dbE\int_0^1 |u_1(t)|^2 dt. $$
Thus, for every $u_1$,
$$ J(0;u_1,0) = \dbE\lt[-|X(1)|^2 + \int_0^1 |u_1(s)|^2 ds\rt] \ges0. $$
When $x=0$ and $u_1=0$, for every $u_2$,
$$ J(0;0,u_2) = -\dbE\lt[|X(1)|^2 + \int_0^1 |u_2(t)|^2 dt\rt] \les0. $$
However, $V^-(x)=-\i$ for every $x\ne0$. To see this, let $u_2$ be an arbitrary control of Player 2
and take $u_1=\l\in\dbR$. Then
$$ \dbE|X(1)|^2 = \dbE\lt[x + \l + \int_0^1 u_2(t) dW(t)\rt]^2
                = (x+\l)^2 + \dbE\int_0^1 |u_2(t)|^2 dt. $$
It follows that
$$ J(x;\l,u_2) = -(x^2+2\l x) - 2\dbE\int_0^1 |u_2(t)|^2 dt. $$
Since $x\ne0$, we see that for every $u_2$,
$$ \inf_{u_1}J(x;u_1,u_2) \les \inf_{\l}J(x;\l,u_2) = -\i, $$
and hence $V^-(x)=\sup_{u_2}\inf_{u_1}J(x;u_1,u_2)=-\i$.
\end{example}

Now we introduce a condition slightly stronger than the necessary conditions \rf{J1>0}--\rf{J2<0}
for the finiteness of the open-loop lower and upper values.
\begin{enumerate}\setlength{\itemindent}{3.7pt}
\item[{\bf\setword{(A3)}{A3}}] There exists a constant $\a>0$ such that
     \begin{alignat}{3}
     J(0;u_1,0) &\ges \a\|u_1\|^2,  \q&& \forall u_1\in\cU_1[0,T], \label{A3-tu}\\
     J(0;0,u_2) &\les -\a\|u_2\|^2, \q&& \forall u_2\in\cU_2[0,T]. \label{A3-ao}
     \end{alignat}
\end{enumerate}
It can be shown that \ref{A3} is a sufficient condition for the finiteness of the open-loop lower
and upper values at every initial state.
In fact, we shall see in the next section that under \ref{A3}, the two-person zero-sum stochastic
LQ differential game even admits an open-loop saddle point for every initial state.
Since the argument involves the Riccati equation, we defer the proof to the next section.
For the moment we want to point out that if the necessary conditions \rf{J1>0}--\rf{J2<0} hold,
then for each $\l>0$, the quadratic functional defined by
$$ J_{\l}(x;u_1,u_2) \deq J(x;u_1,u_2) +\l\dbE\int_0^T|u_1(t)|^2dt -\l\dbE\int_0^T|u_2(t)|^2dt $$
satisfies \ref{A3}. Let $V_{\l}^-(x)$ and $V_{\l}^+(x)$ be the open-loop lower and upper values
corresponding to the quadratic functional $ J_{\l}(x;u_1,u_2)$, respectively, i.e.,
\begin{align*}
V_{\l}^-(x) &\deq \sup_{u_2\in\cU_2[0,T]} \inf_{u_1\in\cU_1[0,T]} J_{\l}(x;u_1,u_2), \\
V_{\l}^+(x) &\deq \inf_{u_1\in\cU_1[0,T]} \sup_{u_2\in\cU_2[0,T]} J_{\l}(x;u_1,u_2).
\end{align*}
We have the following result.

\begin{proposition}\label{prop:Vl-V}
Let {\rm\ref{A1}--\ref{A2}} hold. If $V^\pm(x)$ are finite, then
\begin{align*}
V^-(x) \les \liminf_{\l\to0} V_{\l}^-(x) \les \limsup_{\l\to0} V_{\l}^+(x) \les V^+(x).
\end{align*}
\end{proposition}

\begin{proof}
The second inequality trivially holds.
So we only prove the first, the last can be treated in a similar manner.
Let $\e>0$ be an arbitrary number and choose an $u_2^{\e}\in\cU_2[0,T]$ such that
\begin{equation*}
V^-(x) = \sup_{u_2\in\cU_2[0,T]} \inf_{u_1\in\cU_1[0,T]} J(x;u_1,u_2)
    \les \inf_{u_1\in\cU_1[0,T]} J(x;u_1,u_2^\e) + \e.
\end{equation*}
With this $u_2^{\e}$ fixed, we have for any $\l>0$,
\begin{align*}
\inf_{u_1\in\cU_1[0,T]} J_\l(x;u_1,u_2^\e)
&= \inf_{u_1\in\cU_1[0,T]} \lt[J(x;u_1,u_2^\e) +\l\dbE\int_0^T|u_1(t)|^2dt -\l\dbE\int_0^T|u_2^\e(t)|^2dt\rt] \\
&\ges \inf_{u_1\in\cU_1[0,T]}J(x;u_1,u_2^\e) -\l\dbE\int_0^T|u_2^\e(t)|^2dt \\
&\ges V^-(x) -\e -\l\dbE\int_0^T|u_2^\e(t)|^2dt,
\end{align*}
from which it follows that
\begin{align*}
V_{\l}^-(x)  &=   \sup_{u_2\in\cU_2[0,T]} \inf_{u_1\in\cU_1[0,T]} J_{\l}(x;u_1,u_2)
            \ges \inf_{u_1\in\cU_1[0,T]} J_\l(x;u_1,u_2^\e) \\
&\ges  V^-(x)-\e -\l\dbE\int_0^T|u_2^\e(t)|^2dt, \q\forall \l>0.
\end{align*}
Letting $\l\to0$ yields
$$ \liminf_{\l\to0}V_{\l}^-(x) \ges V^-(x)-\e. $$
Since $\e>0$ is arbitrary, the desired result follows.
\end{proof}

%
%
%

We conclude this section with a discussion of the conditions \rf{J1>0}--\rf{J2<0} and \ref{A3}.
Consider the stochastic LQ optimal control problem with the state equation
\begin{equation*}
\left\{\begin{aligned}
dX(t) &= \big[A(t)X(t)+B_1(t)v(t)\big]dt + \big[C(t)X(t)+D_1(t)v(t)\big]dW(t), \\
 X(0) &= x,
\end{aligned}\right.
\end{equation*}
and the cost functional
\begin{equation*}
\cJ_1(x;v) \deq \dbE\Bigg\{\lan GX(T),X(T)\ran +\int_0^T\Blan\!
   \begin{pmatrix}Q(t)   & \!S_1(t)^\top \\
                  S_1(t) & \!R_{11}(t)   \end{pmatrix}\!
   \begin{pmatrix}X(t) \\ v(t) \end{pmatrix}\!,
   \begin{pmatrix}X(t) \\ v(t) \end{pmatrix}\! \Bran dt\Bigg\}.
\end{equation*}
Let us denote the above optimal control problem by Problem (SLQ)$_1$. Clearly,
$$ \cJ_1(x;u_1)=J(x;u_1,0), \q\forall u_1\in\cU_1[0,T]. $$
So the condition \rf{J1>0} is equivalent to
\begin{equation}\label{SLQ1:convex}
\cJ_1(0;v)\ges0, \q\forall v\in L_\dbF^2(0,T;\dbR^{m_1}),
\end{equation}
which means that the mapping $v\mapsto\cJ_1(0;v)$ is convex, i.e.,
\begin{align*}
\cJ_1(0;\b u+(1-\b)v) \les \b\cJ_1(0;u) + (1-\b)\cJ_1(0;v), \\
\forall u,v\in L_\dbF^2(0,T;\dbR^{m_1}),\forall \b\in[0,1].
\end{align*}
Similarly, the condition \rf{A3-tu} in \ref{A3} is equivalent to the uniform convexity of $v\mapsto\cJ_1(0;v)$.
Likewise, if we consider the stochastic LQ optimal control problem with the state equation
\begin{equation*}
\left\{\begin{aligned}
dX(t) &= \big[A(t)X(t)+B_2(t)v(t)\big]dt + \big[C(t)X(t)+D_2(t)v(t)\big]dW(t), \\
 X(0) &= x,
\end{aligned}\right.
\end{equation*}
and the cost functional
\begin{equation*}
\cJ_2(x;v) \deq -\dbE\Bigg\{\lan GX(T),X(T)\ran +\int_0^T\Blan\!
   \begin{pmatrix}Q(t)   & \!S_2(t)^\top \\
                  S_2(t) & \!R_{22}(t)   \end{pmatrix}\!
   \begin{pmatrix}X(t) \\ v(t) \end{pmatrix}\!,
   \begin{pmatrix}X(t) \\ v(t) \end{pmatrix}\! \Bran dt\Bigg\},
\end{equation*}
which we denote by Problem (SLQ)$_2$ for simplicity, then the condition \rf{J2<0} is equivalent to
\begin{equation}\label{SLQ2:convex}
\cJ_2(0;v)\ges0, \q\forall v\in L_\dbF^2(0,T;\dbR^{m_2}),
\end{equation}
and the condition \rf{A3-ao} in \ref{A3} is equivalent to the uniform convexity of $v\mapsto\cJ_2(0;v)$.
There are various sufficient conditions ensuring the (uniform) convexity of the cost functional of
a stochastic LQ optimal control problem. For results in this direction we refer the interested reader
to \cite{Sun-Li-Yong2016,Sun-Yong2020,Sun-Xiong-Yong2020}.

\section{Open-loop saddle points and Riccati equations}\label{Sec:Saddle}

The aim of this section is to show the existence of open-loop saddle points and to provide
a closed-loop representation for open-loop saddle points under the condition \ref{A3}.
The associated Riccati equation plays a crucial role in establishing these results, whose
solvability constitutes the most difficult part of this section.

\ms

We begin by recalling the notion of open-loop saddle points and introducing the Riccati equation.

\begin{definition}
An {\it open-loop saddle point} for the initial state $x$ is a pair $(u^*_1,u^*_2)\in\cU_1[0,T]\times\cU_2[0,T]$
such that the following inequalities hold:
\begin{equation*}
J(x;u^*_1,u_2) \les J(x;u^*_1,u^*_2) \les J(x;u_1,u^*_2), \q\forall (u_1,u_2)\in\cU_1[0,T]\times\cU_2[0,T].
\end{equation*}
\end{definition}

The Riccati equation associated with the two-person zero-sum stochastic LQ differential game is
a nonlinear ordinary differential equation of the following form:
\begin{equation}\label{Ric}\left\{\begin{aligned}
 & \dot P + PA + A^\top P + C^\top PC + Q\\
 & \hp{\dot P} -(PB+C^\top PD+S^\top)(R+D^\top PD)^{-1}(B^\top P+D^\top PC+S)=0, \\
 & P(T)=G,
\end{aligned}\right.\end{equation}
where we have adopted the notation
$$ B = (B_1, B_2), \q D = (D_1, D_2), \q S = \begin{pmatrix}S_1 \\ S_2 \end{pmatrix},
\q R = \begin{pmatrix}R_{11} & R_{12} \\
                      R_{21} & R_{22} \end{pmatrix}, $$
and as before, the variable $t$ has been suppressed for convenience. Note that
\begin{align*}
& R+D^\top PD = \begin{pmatrix}R_{11} + D_1^\top PD_1 & R_{12} + D_1^\top PD_2 \\[1mm]
                               R_{21} + D_2^\top PD_1 & R_{22} + D_2^\top PD_2 \end{pmatrix}, \\
& B^\top P+D^\top PC+S = \begin{pmatrix}B_1^\top P + D_1^\top PC + S_1 \\[1mm]
                                        B_2^\top P + D_2^\top PC + S_2 \end{pmatrix}.
\end{align*}

\begin{definition}
A {\it strongly regular solution} to the Riccati equation \rf{Ric} over $[0,T]$ is an absolutely
continuous function $P:[0,T]\to\dbS^n$ that possesses the following properties:
\begin{enumerate}[(i)]\setlength{\itemsep}{-1pt}
\item For $i=1,2$, $(-1)^{i+1}(R_{ii} + D_i^\top PD_i) \gg0$.
\item $P$ satisfies \rf{Ric} almost everywhere on $[0,T]$.
\end{enumerate}
\end{definition}

Now we state the main results of this section. The proofs will be given shortly after some preparations.

\begin{theorem}\label{thm:main-Ric}
Let {\rm\ref{A1}--\ref{A3}} hold.
Then the Riccati equation \rf{Ric} has a strongly regular solution over $[0,T]$.
\end{theorem}

\begin{theorem}\label{thm:main-saddle}
Let {\rm\ref{A1}--\ref{A3}} hold, and let $P\in C([0,T];\dbS^n)$ be the strongly regular
solution to the Riccati equation \rf{Ric} over $[0,T]$. Then
\begin{enumerate}[\rm(i)]\setlength{\itemsep}{-1pt}
\item for every initial state $x$, a unique open-loop saddle point exists;
\item with the notation
      $$ \Th = -(R+D^\top PD)^{-1}(B^\top P+D^\top PC+S), $$
      the open-loop saddle point $u^*=\begin{pmatrix}u_1^*\\u_2^*\end{pmatrix}$ for the initial state $x$
      has the following closed-loop representation:
      \begin{align}\label{u:cloop-rep}
       u^*(t) = \Th(t)X^*(t), \q t\in[0,T],
      \end{align}
      where $X^*$ is the solution to the closed-loop system
      \begin{equation}\label{state:cloop}\left\{\begin{aligned}
      dX^*(t) &= [A(t)+B(t)\Th(t)]X^*(t)dt + [C(t)+D(t)\Th(t)]X^*(t)dW(t), \\
       X^*(0) &= x.
      \end{aligned}\right.\end{equation}
      Moreover, $J(x;u_1^*,u_2^*)=\lan P(0)x,x\ran$.
\end{enumerate}
\end{theorem}

It is worth pointing out that the converse of \autoref{thm:main-Ric} does not hold in general.
In other words, the existence of a strongly regular solution to the Riccati equation \rf{Ric}
over $[0,T]$ does not necessarily imply the condition \ref{A3}.
In fact, the existence of a strongly regular solution does not even imply the weaker conditions
\rf{J1>0}--\rf{J2<0}. Here is an example.

\begin{example}\label{ex:PnotA3}
Consider the one-dimensional state equation
$$\left\{\begin{aligned}
dX(t) &= [u_1(t)dt + u_2(t)] dt,\q t\in[0,1], \\
 X(0) &= x,
\end{aligned}\right.$$
and the quadratic functional
$$ J(x;u_1,u_2) = \dbE\lt\{ -2|X(1)|^2 + \int_0^1 \lt[|u_1(t)|^2-{2\over3}|u_2(t)|^2\rt] dt\rt\}. $$
The Riccati equation associated with this game reads
$$\left\{\begin{aligned}
 & \dot P -P(1,1)\begin{pmatrix*}[r]1&0\\0&-{2\over3}\end{pmatrix*}^{-1}\begin{pmatrix}1\\1\end{pmatrix}P=0, \\
 & P(1)=-2,
\end{aligned}\right.$$
which simplifies to
\begin{equation}\label{ex4.5:Ric}\left\{\begin{aligned}
 & \dot P(t) =-{1\over2}P(t)^2, \\
 & P(1)=-2.
\end{aligned}\right.\end{equation}
It is straightforward to verify that
$$ P(t) = {2\over t-2}, \q t\in[0,1] $$
is the strongly regular solution of \rf{ex4.5:Ric} over $[0,1]$. However, for $u_1(t)\equiv\l\ne0$, we have
$$ J(0;u_1,0) = -2\lt(\int_0^1\l dt\rt)^2 + \int_0^1\l^2 dt = -\l^2<0. $$
\end{example}

The preparation for the proof of \autoref{thm:main-Ric} starts with the following lemma,
whose proof is straightforward.

\begin{lemma}\label{lmm:matrix}
For $M\in\dbS^m$, $L\in\dbR^{m\times n}$, $N\in\dbS^n$,
if $M$ and $\F\deq N-L^\top M^{-1}L$ are invertible, then $\begin{pmatrix}M&L \\ L^\top&N\end{pmatrix}$
is also invertible and
$$ \begin{pmatrix}M & L \\ L^\top & N\end{pmatrix}^{-1}
 = \begin{pmatrix}M^{-1} + (M^{-1}L) \F^{-1} (M^{-1}L)^\top & -(M^{-1}L) \F^{-1} \\
                                  -\F^{-1} (M^{-1}L)^\top & \F^{-1}          \end{pmatrix}. $$
Moreover, for every $\rho\in\dbR^{m\times k}$ and $\xi\in\dbR^{n\times k}$,
$$ (\rho^\top, \xi^\top)\begin{pmatrix}M&L \\ L^\top&N\end{pmatrix}^{-1}\begin{pmatrix}\rho \\ \xi\end{pmatrix}
 = \rho^\top M^{-1}\rho + \big(L^\top M^{-1}\rho -\xi\big)^\top\F^{-1}\big(L^\top M^{-1}\rho -\xi\big). $$
In particular, if $M$ is positive definite and $N$ is negative definite, then
$$ (\rho^\top, \xi^\top)\begin{pmatrix}M&L \\ L^\top&N\end{pmatrix}^{-1}\begin{pmatrix}\rho\\ \xi\end{pmatrix}
\les \rho^\top M^{-1}\rho, \q\forall \rho\in\dbR^{m\times k},~\xi\in\dbR^{n\times k}. $$
\end{lemma}

We next make some observations. Suppose that \ref{A3} holds. Then the cost functional $\cJ_1(x;v)$
of Problem (SLQ)$_1$ introduced in the preceding section satisfies
$$ \cJ_1(0;v) \ges \a\|v\|^2, \q\forall v\in  L_\dbF^2(0,T;\dbR^{m_1}). $$
Thus, by \autoref{lmm:SLQ-main-2}(ii), the Riccati equation
\begin{equation}\label{Ric:1}\left\{\begin{aligned}
 & \dot P_1 + P_1A + A^\top P_1 + C^\top P_1C + Q\\
 & \hp{\dot P_1} -(P_1B_1+C^\top P_1D_1+S_1^\top)(R_{11}+D_1^\top P_1D_1)^{-1}(B_1^\top P_1+D_1^\top P_1C+S_1)=0, \\
 & P_1(T)=G
\end{aligned}\right.\end{equation}
admits a unique solution $P_1\in C([0,T];\dbS^n)$ satisfying
\begin{equation}\label{R11+D1P1D1>>0}
  R_{11}+D_1^\top P_1D_1\gg0.
\end{equation}
Likewise, the Riccati equation
\begin{equation}\label{Ric:2}\left\{\begin{aligned}
&\dot P_2 + P_2A + A^\top P_2 + C^\top P_2C + Q\\
&\hp{\dot P_2} -(P_2B_2+C^\top P_2D_2+S_2^\top)(R_{22}+D_2^\top P_2D_2)^{-1}(B_2^\top P_2+D_2^\top P_2C+S_2)=0, \\
& P_2(T)=G
\end{aligned}\right.\end{equation}
admits a unique solution $P_2\in C([0,T];\dbS^n)$ satisfying
\begin{equation}\label{R22+D2P2D2<<0}
  R_{22}+D_2^\top P_2D_2 \ll 0.
\end{equation}
We have the following comparison result.

\begin{proposition}\label{prop:comparison}
Let {\rm\ref{A1}--\ref{A3}} hold.
Suppose that $P$ is a strongly regular solution of \rf{Ric} over some interval $[\th,\t]\subseteq[0,T]$
with terminal condition replaced by $P(\t)=H\in\dbS^n$. If $P_1(\t)\les H\les P_2(\t)$, then
$$ P_1(t) \les P(t) \les P_2(t), \q\forall t\in[\th,\t]. $$
\end{proposition}

\begin{proof}
First we introduce the following notation: For a matrix $K\in\dbS^n$,
\begin{eqnarray*}
&\begin{pmatrix}\cM(t,K)      & \cL(t,K) \\[1mm]
                \cL(t,K)^\top & \cN(t,K) \end{pmatrix}
 \deq \begin{pmatrix}R_{11}(t) + D_1(t)^\top KD_1(t) & R_{12}(t) + D_1(t)^\top KD_2(t) \\[1mm]
                     R_{21}(t) + D_2(t)^\top KD_1(t) & R_{22}(t) + D_2(t)^\top KD_2(t) \end{pmatrix}, \\
&\cR(t,K) \deq \begin{pmatrix}\cM(t,K)      & \cL(t,K) \\[1mm]
                              \cL(t,K)^\top & \cN(t,K) \end{pmatrix},  \\
&\cS_i(t,K) \deq B_i(t)^\top K + D_i(t)^\top KC(t) + S_i(t), \q i=1,2, \\
&\cS(t,K) \deq \begin{pmatrix}\cS_1(t,K) \\[1mm]
                              \cS_2(t,K) \end{pmatrix} =  B(t)^\top K + D(t)^\top KC(t) + S(t).
\end{eqnarray*}
By \autoref{lmm:matrix}, we have for $t\in[\th,\t]$,
\begin{align*}
& [(PB+C^\top PD+S^\top)(R+D^\top PD)^{-1}(B^\top P+D^\top PC+S)](t) \\
&\q\les \cS_1(t,P(t))^\top\cM(t,P(t))^{-1}\cS_1(t,P(t)).
\end{align*}
Now let
$$ \Pi(t)=P(t)-P_1(t), \q t\in[\th,\t]. $$
Then $\Pi(\t)\ges0$ and
\begin{align*}
\dot\Pi(t)
&\les -\,[\Pi(t)A(t) + A(t)^\top\Pi(t) + C(t)^\top\Pi(t)C(t)] \\
&~\hp{=} + \cS_1(t,P(t))^\top\cM(t,P(t))^{-1}\cS_1(t,P(t)).
\end{align*}
With the notation
$$ \h S \deq B_1^\top P_1+D_1^\top P_1C+S_1, \q \h R= R_{11}+D_1^\top P_1D_1, $$
we can rewrite $\cS_1(t,P(t))^\top\cM(t,P(t))^{-1}\cS_1(t,P(t))$ as
\begin{align*}
& \cS_1(t,P(t))^\top\cM(t,P(t))^{-1}\cS_1(t,P(t)) \\
&\q=\big[(PB_1+C^\top PD_1+S_1^\top)(R_{11}+D_1^\top PD_1)^{-1}(B_1^\top P+D_1^\top PC+S_1)\big](t) \\
&\q=\big[(\Pi B_1+C^\top \Pi D_1+\h S^{\,\top})(\h R+D_1^\top\Pi D_1)^{-1}
         (B_1^\top\Pi+D_1^\top\Pi C+\h S\,)\big](t).
\end{align*}
It follows that for some $\h Q\in L^\i(\th,\t;\dbS^n)$ with $\h Q\ges0$,
\begin{align*}
&\dot\Pi + \Pi A + A^\top\Pi + C^\top\Pi C + \h Q \\
&\hp{\dot\Pi} -(\Pi B_1+C^\top \Pi D_1+\h S^{\,\top})(\h R+D_1^\top\Pi D_1)^{-1}
               (B_1^\top\Pi+D_1^\top\Pi C+\h S\,) =0.
\end{align*}
Since $\Pi(\t)\ges0$, $\h Q\ges0$, and $\h R\gg0$, we conclude from \autoref{lmm:SLQ-main-2}(ii)
that $\Pi(t)\ges0$ for all $t\in[\th,\t]$. This shows that $P_1\les P$ on $[\th,\t]$.
In a similar manner we can prove that $P\les P_2$ on $[\th,\t]$.
\end{proof}

The following result establishes the local existence of a strongly regular solution
to the Riccati equation \rf{Ric}.

\begin{proposition}\label{prop:local-sol-Ric}
Let {\rm\ref{A1}--\ref{A3}} hold, and let $P_1$ and $P_2$ be the solutions of
\rf{Ric:1} and \rf{Ric:2}, respectively. For $\t\in(0,T]$ and $H\in\dbS^n$, if
$$ P_1(\t) \les H \les P_2(\t), $$
then the Riccati equation
\begin{equation}\label{Ric:(tau,H)}\left\{\begin{aligned}
 & \dot P + PA + A^\top P + C^\top PC + Q\\
 & \hp{\dot P} -(PB+C^\top PD+S^\top)(R+D^\top PD)^{-1}(B^\top P+D^\top PC+S)=0, \\
 & P(\t)=H
\end{aligned}\right.\end{equation}
is locally solvable at $\t$, that is, for $\e>0$ small enough, \rf{Ric:(tau,H)} has
a strongly regular solution on $[\t-\e,\t]$.
\end{proposition}

\begin{proof}
We have seen that for some constant $\a>0$,
\begin{align}\label{R11>0-R22<0}
R_{11}(t)+D_1(t)^\top P_1(t)D_1(t) \ges  \a I_{m_1}, \q
R_{22}(t)+D_2(t)^\top P_2(t)D_2(t) \les -\a I_{m_2},
\end{align}
for almost every $t\in[0,T]$. Since changing the values of $R_{ii}$ and $D_i$ ($i=1,2$) on a set
of Lebesgue measure zero does not affect the solvability of the Riccati equation \rf{Ric:(tau,H)},
we may assume without loss of generality that \rf{R11>0-R22<0} holds for all $t\in[0,T]$. Let us
denote by $\|D\|_\i$ the essential supremum of $D=(D_1, D_2)\in L^\i(0,T;\dbR^{n\times (m_1+m_2)})$,
and let
$$ r={\a\over4(\|D\|_\i^2+1)}. $$
Since $P_1$ and $P_2$ are continuous, we can choose a small $\d>0$ such that
$$ |P_i(t)-P_i(\t)| \les r, \q\forall t\in[\t-\d,\t],~i=1,2. $$
Denote by $\cB_r(H)$ the closed ball in $\dbS^n$ with center $H$ and radius $r$.
Then for any $t\in[\t-\d,\t]$ and $M\in\cB_r(H)$,
\begin{align}\label{Br:1}
R_{11}(t)+D_1(t)^\top MD_1(t)
  &=    R_{11}(t)+D_1(t)^\top HD_1(t) + D_1(t)^\top(M-H)D_1(t)  \nn\\
  &\ges R_{11}(t)+D_1(t)^\top P_1(\t)D_1(t) - \|D\|_\i^2|M-H|I_{m_1} \nn\\
  &\ges R_{11}(t)+D_1(t)^\top P_1(t)D_1(t) - \|D\|_\i^2|P_1(\t)-P_1(t)|I_{m_1}  \nn\\
  &~\hp{=} - \|D\|_\i^2|M-H|I_{m_1} \nn\\
  &\ges \a I_{m_1} - 2r\|D\|_\i^2 I_{m_1} \nn\\
  &\ges {\a\over2} I_{m_1}.
\end{align}
Similarly, for any $t\in[\t-\d,\t]$ and $M\in\cB_r(H)$,
\begin{align}\label{Br:2}
R_{22}(t)+D_2(t)^\top MD_2(t) \les -{\a\over2}I_{m_2}.
\end{align}
From \rf{Br:1} and \rf{Br:2} we conclude that the function
$$ F:[0,T]\times\dbS^n \to \dbS^n $$
defined by (recalling the notation introduced in the proof of \autoref{prop:comparison})
\begin{align*}
F(t,P) &= PA(t) + A(t)^\top P + C(t)^\top PC(t) + Q(t)-\cS(t,P)^\top\cR(t,P)^{-1}\cS(t,P)
\end{align*}
is Lipschitz continuous in $P$ on $[\t-\d,\t]\times\cB_r(H)$, that is, there exists a
constant $\rho>0$ such that
$$ |F(t,P)-F(t,M)| \les \rho|P-M|, \q\forall t\in[\t-\d,\t],~\forall P,M\in\cB_r(H). $$
Indeed, we see from \rf{Br:1} and \rf{Br:2} that $\cM(t,P)$ and $\cN(t,P)$ are invertible
for every $(t,P)\in[\t-\d,\t]\times\cB_r(H)$ with
$$ |\cM(t,P)^{-1}| \les {2\over\a}\sqrt{m_1}, \q |\cN(t,P)^{-1}| \les {2\over\a}\sqrt{m_2}. $$
Moreover, since $\F(t,P)\deq\cN(t,P)-\cL(t,P)^\top\cM(t,P)^{-1}\cL(t,P)\les \cN(t,P)$, we have
\begin{align*}
\big|\F(t,P)^{-1}\big| \les {2\over\a}\sqrt{m_2}, \q\forall (t,P)\in[\t-\d,\t]\times\cB_r(H).
\end{align*}
Since the coefficients of the state equation and the weighting matrices in the cost functional are bounded,
we can choose a constant $\rho>0$ such that
$$ |\cL(t,P)|+|\cS(t,P)| \les \rho, \q\forall (t,P)\in[\t-\d,\t]\times\cB_r(H). $$
For convenience, in the sequel we shall use the same letter $\rho$ to denote constants independent
of $(t,P)\in[\t-\d,\t]\times\cB_r(H)$. Then we have by \autoref{lmm:matrix},
\begin{align*}
\big|\cR(t,P)^{-1}\big|^2
&= \lt|\begin{pmatrix}\cM(t,P)      & \cL(t,P) \\[1mm]
                      \cL(t,P)^\top & \cN(t,P) \end{pmatrix}^{-1}\rt|^2 \\
&= |\cM(t,P)^{-1} + \cM(t,P)^{-1}\cL(t,P)\F(t,P)^{-1} \cL(t,P)^\top\cM(t,P)^{-1}|^2  \\
&~\hp{=}+ 2|\cM(t,P)^{-1}\cL(t,P)\F(t,P)^{-1}|^2 + |\F(t,P)^{-1}|^2 \\
&\les \rho, \q\forall (t,P)\in[\t-\d,\t]\times\cB_r(H).
\end{align*}
Noting that for any $P,M\in\dbS^n$,
\begin{align*}
&\cS(t,P)^\top\cR(t,P)^{-1}\cS(t,P) - \cS(t,M)^\top\cR(t,M)^{-1}\cS(t,M) \\
&\q= [\cS(t,P)-\cS(t,M)]^\top\cR(t,P)^{-1}\cS(t,P) \\
&\q~\hp{=} +\cS(t,M)^\top\cR(t,P)^{-1}[\cS(t,P)-\cS(t,M)] \\
&\q~\hp{=} +\cS(t,M)^\top\cR(t,P)^{-1}D(t)^\top(M-P)D(t)\cR(t,M)^{-1}\cS(t,M),
\end{align*}
and that
$$ |\cS(t,P)-\cS(t,M)| \les \rho|P-M|, \q\forall t\in[\t-\d,\t],~\forall P,M\in\cB_r(H), $$
we obtain the Lipschitz continuity of $F$ in $P$ by computing $|F(t,P)-F(t,M)|$ directly.
Thanks to the Lipschitz continuity of $F$, the existence of a strongly regular solution on a
small interval $[\t-\e,\t]$ follows by the usual Picard's iteration method (or equivalently,
by the contraction mapping theorem).
\end{proof}

We are now ready to give the proof of \autoref{thm:main-Ric}.

\begin{proof}[\textbf{Proof of \autoref{thm:main-Ric}}]
(i) Suppose that \ref{A3} holds for some constant $\a>0$. Then by \autoref{prop:local-sol-Ric},
the Riccati equation \rf{Ric} is locally solvable at $T$.
We show that the local solution of \rf{Ric} can be extended to $[0,T]$.
To this end, let $(\t,T]$ be the maximal interval on which a strongly regular solution $P$ of
\rf{Ric} exists. By \autoref{prop:comparison},
\begin{align}\label{P1<P<P2}
 P_1(t) \les P(t) \les P_2(t), \q\forall t\in(\t,T].
\end{align}
It follows that the function
\begin{align*}
F(t,P(t)) &\deq P(t)A(t) + A(t)^\top P(t) + C(t)^\top P(t)C(t) + Q(t) \\
&~\hp{=} -\cS(t,P(t))^\top\cR(t,P(t))^{-1}\cS(t,P(t))
\end{align*}
is bounded on $(\t,T]$ and hence
$$ P(t) = G + \int_t^T F(s,P(s))ds, \q t\in(\t,T] $$
is uniformly continuous. Thus, the limit $\lim_{t\to\t}P(t)$ exists and is finite, and thereby
we can extend $P$ to the closed interval $[\t,T]$ by setting
$$ P(\t) = \lim_{t\to\t}P(t). $$
Note that \rf{P1<P<P2} implies
$$ P_1(\t)\les P(\t)\les P_2(\t). $$
If $\t>0$, then by \autoref{prop:local-sol-Ric}, the solution $P$ can be further extended to an
interval larger than $[\t,T]$. This contradicts the maximality of $[\t,T]$. So we must have $\t=0$.
\end{proof}

In order to prove \autoref{thm:main-saddle}, we need the following lemma,
whose proof can be found in \cite{Sun-Yong2014}.

\begin{lemma}\label{lmm:Sun2014}
Let {\rm\ref{A1}--\ref{A2}} hold.
A pair $(u_1^*,u_2^*)\in\cU_1[0,T]\times\cU_2[0,T]$ is an open-loop saddle point for the initial state $x$
if and only if \rf{J1>0}--\rf{J2<0} hold and with $u^*\deq\begin{pmatrix}u_1^*\\u_2^*\end{pmatrix}$,
\begin{align}\label{Sun14-Biyao}
B^\top Y^* + D^\top Z^* + SX^* + Ru^* = 0, \q \ae~\hb{on}~[0,T],~\as,
\end{align}
where $(X^*,Y^*,Z^*)$ is the adapted solution to the following decoupled forward-backward stochastic
differential equation (FBSDE, for short):
\begin{equation}\label{FBSDE}\left\{\begin{aligned}
dX^*(t) &= [A(t)X^*(t)+B(t)u^*(t)]dt + [C(t)X^*(t)+D(t)u^*(t)]dW, \\
dY^*(t) &= -[A(t)^\top Y^*(t)+C(t)^\top\! Z^*(t)+Q(t)X^*(t)+S(t)^\top\! u^*(t)]dt + Z^*(t)dW, \\
 X^*(0) &= x, \q Y^*(T)=GX^*(T).
\end{aligned}\right.\end{equation}
\end{lemma}

\begin{proof}[\textbf{Proof of \autoref{thm:main-saddle}}]
We first prove the uniqueness of an open-loop saddle point.
Suppose that for some $x$, the game has two open-loop saddle points $u^*$ and $v^*$.
Then by \autoref{lmm:Sun2014}, $\bar u =\begin{pmatrix}\bar u_1\\ \bar u_2\end{pmatrix} \deq u^*-v^*$
is an open-loop saddle point for the initial state $0$.
From \ref{A3}, we see that $(0,0)$ is also an open-loop saddle point for the initial state $0$, since
$$J(0;0,u_2) \les J(0;0,0)=0\les J(0;u_1,0), \q \forall (u_1,u_2)\in\cU_1[0,T]\times\cU_2[0,T]. $$
Thus, by the definition of an open-loop saddle point and \ref{A3},
$$ \a\|\bar u_1\|^2 \les J(0;\bar u_1,0) \les J(0;\bar u_1,\bar u_2) \les J(0;0,\bar u_2) \les J(0;0,0) =0, $$
from which we obtain $\bar u_1=0$. Similarly, we can show $\bar u_2=0$. The uniqueness follows.

\ms

In order to prove the existence of an open-loop saddle point and part (ii),
according to \autoref{lmm:Sun2014} it suffices to show that with $u^*$
defined by \rf{u:cloop-rep}, the adapted solution of \rf{FBSDE} satisfies \rf{Sun14-Biyao}.
This can be accomplished by verifying that $(Y^*,Z^*)$ defined by
\begin{align}\label{Y=PX}
Y^* = PX^*, \q Z^*= P(CX^*+Du^*)
\end{align}
is the adapted solution to the BSDE in \rf{FBSDE}. Indeed, integration by parts yields
\begin{align*}
dY^* &= \dot PX^*dt + P(A+B\Th)X^*dt + P(C+D\Th)X^*dW \\
&= -[A^\top P + C^\top PC + Q - \Th^\top\cR(P)\Th- PB\Th]X^*dt + P(C+D\Th)X^*dW \\
&= -[A^\top P + C^\top PC + Q + \cS(P)^\top\Th - PB\Th]X^*dt + P(CX^*+Du^*)dW \\
&= -[A^\top P + C^\top PC + Q + (C^\top PD + S^\top)\Th]X^*dt + Z^*dW \\
&= -[A^\top Y^* + C^\top P(CX^*+Du^*) + QX^* + S^\top u^*]dt + Z^*dW \\
&= -[A^\top Y^* + C^\top Z^* + QX^* + S^\top u^*]dt + Z^*dW.
\end{align*}
On the other hand, $Y^*(T)=P(T)X^*(T)=GX^*(T)$. So with $u^*$ defined by \rf{u:cloop-rep}, the solution $X^*$ to \rf{state:cloop} and $(Y^*,Z^*)$ defined by \rf{Y=PX} satisfy the FBSDE \rf{FBSDE}. Furthermore,
\begin{align*}
B^\top Y^* + D^\top Z^* + SX^* + Ru^*
&= B^\top PX^* + D^\top P(C+D\Th)X^* + SX^* + R\Th X^* \\
&= [B^\top P + D^\top PC+ S + (R+D^\top PD)\Th]X^* \\
&= 0.
\end{align*}
Finally, by integration by parts, we have
\begin{align}\label{GX(T)X(T)}
&\dbE\lan GX^*(T),X^*(T)\ran = \dbE\lan Y^*(T),X^*(T)\ran \nn\\
&\q= \dbE\lan Y^*(0),X^*(0)\ran + \dbE\int_0^T\[\lan B^\top Y^*+D^\top Z^*-SX^*,u^*\ran - \lan QX^*,X^*\ran\] dt.
\end{align}
Substituting \rf{GX(T)X(T)} into
\begin{align*}
J(x;u^*) = \dbE\bigg\{\lan GX^*(T),X^*(T)\ran
+\! \int_0^T\!\[\lan QX^*,X^*\ran + 2\lan SX^*,u^*\ran + \lan Ru^*,u^*\ran\]dt\bigg\},
\end{align*}
and noting that
$$ Y^* = PX^*, \q B^\top Y^* + D^\top Z^* + SX^* + Ru^* = 0, $$
we obtain
\begin{align*}
J(x;u^*) = \dbE\lan Y^*(0),X^*(0)\ran + \dbE\int_0^T \lan B^\top Y^*+D^\top Z^*+SX^*+Ru^*,u^*\ran dt = \lan P(0)x,x\ran.
\end{align*}
This completes the proof.
\end{proof}

\section{Relation between the open-loop saddle point and the open-loop lower and upper values}
\label{Sec:Relation}

In this section we investigate the connection between the open-loop saddle point and
the open-loop lower and upper values.
We shall first show that the existence of an open-loop saddle point implies the existence
of a finite open-loop value and hence the finiteness of the open-loop lower and upper values,
but not vice versa in general.
Then, for the deterministic two-person zero-sum LQ differential game, we give an alternative
proof for Zhang's result \cite{Zhang2005} on the equivalence of the existence of an open-loop
saddle point and the finiteness of the open-loop lower and upper values.
In particular, we show that in the deterministic case, the finiteness of the open-loop lower
and upper values implies the solvability of the Riccati equation, for which a fairly explicit
representation of the solution can be obtained.

\begin{proposition}\label{prop:saddle-value}
Let {\rm\ref{A1}--\ref{A2}} hold. If an open-loop saddle point $(u^*_1,u^*_2)$ exists for
the initial state $x$, then Problem (SLQG) admits a finite open-loop value at $x$ and
$$ V(x) = J(x;u^*_1,u^*_2). $$
\end{proposition}

\begin{proof}
Suppose that $(u^*_1,u^*_2)$ is an open-loop saddle point for $x$. Then
\begin{align*}
J(x;u^*_1,u^*_2)
  & \les \inf_{u_1\in\cU_1[0,T]} J(x;u_1,u^*_2)
    \les \sup_{u_2\in\cU_2[0,T]} \inf_{u_1\in\cU_1[0,T]} J(x;u_1,u_2) = V^-(x), \\
J(x;u^*_1,u^*_2)
  & \ges \sup_{u_2\in\cU_2[0,T]} J(x;u_1^*,u_2)
    \ges \inf_{u_1\in\cU_1[0,T]} \sup_{u_2\in\cU_2[0,T]} J(x;u_1,u_2) = V^+(x).
\end{align*}
It follows that
$$ V^+(x) \les J(x;u^*_1,u^*_2) \les V^-(x). $$
On the other hand, $V^-(x)\les V^+(x)$. Therefore, equalities hold in the above.
\end{proof}

\autoref{prop:saddle-value} shows that the existence of an open-loop saddle point implies
the finiteness of the open-loop lower and upper values.
However, the converse is not necessarily true in general. Here is an example.

\begin{example}\label{exm:saddle-novuale}
Consider the one-dimensional state equation
$$\left\{\begin{aligned}
dX(t) &= u_1(t)dt + u_2(t) dW(t),\q t\in[0,1], \\
 X(0) &= x,
\end{aligned}\right.$$
and the quadratic functional
$$ J(x;u_1,u_2) = \dbE\bigg\{ |X(1)|^2 + \int_0^1 \[t^2|u_1(t)|^2-|u_2(t)|^2\] dt\bigg\}. $$
For the lower value, we have
$$ V^-(x) \ges \inf_{u_1\in\cU_1[0,T]} J(x;u_1,0)
             = \inf_{u_1\in\cU_1[0,T]} \dbE\bigg\{ |X(1)|^2 + \int_0^1 t^2|u_1(t)|^2 dt\bigg\} \ges0. $$
For the upper value, we have
\begin{align*}
V^+(x) &\les \sup_{u_2\in\cU_2[0,T]} J(x;0,u_2)
           = \sup_{u_2\in\cU_2[0,T]}\dbE\bigg\{|X(1)|^2-\int_0^1|u_2(t)|^2dt\bigg\} \\
       &= \sup_{u_2\in\cU_2[0,T]}\dbE\bigg\{\Big|x+\int_0^1u_2(t)dW(t)\Big|^2-\int_0^1|u_2(t)|^2dt\bigg\} \\
       &= x^2.
\end{align*}
Thus, both the open-loop lower and upper values are finite.
Next we show by contradiction that an open-loop saddle point does not exist for any $x\ne0$.
If $(u_1^*,u_2^*)$ is an open-loop saddle point for some $x\ne0$, then by \autoref{lmm:Sun2014},
the adapted solution $(X^*,Y^*,Z^*)$ to the FBSDE
$$\left\{\begin{aligned}
dX^*(t) &= u_1^*(t)dt + u_2^*(t) dW(t), \q t\in[0,1], \\
dY^*(t) &= Z^*(t) dW(t), \q t\in[0,1], \\
 X^*(0) &= x, \q Y^*(1) = X^*(1)
\end{aligned}\right.$$
should satisfy the following conditions:
\begin{align*}
Y^*(t) + t^2u_1(t) &= 0, \q\ae~t\in[0,1],~\as, \\
Z^*(t) - u_2(t)    &= 0, \q\ae~t\in[0,1],~\as
\end{align*}
By taking expectations, we have
\begin{equation}\label{eqn:EXEY}\left\{\begin{aligned}
d\dbE[X^*(t)] &= \dbE[u_1^*(t)]dt, \q t\in[0,1], \\
d\dbE[Y^*(t)] &= 0, \q t\in[0,1], \\
 \dbE[X^*(0)] &= x, \q \dbE[Y^*(1)] = \dbE[X^*(1)],
\end{aligned}\right.\end{equation}
and
\begin{equation}\label{eqn:E-stationary}
  \dbE[Y^*(t)] + t^2\dbE[u_1(t)] = 0, \q\ae~t\in[0,1].
\end{equation}
From \rf{eqn:EXEY} we see that
$$ \dbE[Y^*(t)] = \dbE[X^*(1)] = x + \int_0^1 \dbE[u_1^*(s)]ds, \q\forall t\in[0,1]. $$
So \rf{eqn:E-stationary} is equivalent to
\begin{equation}\label{eqn:E-stationary*}
  \dbE[X^*(1)] + t^2\dbE[u_1(t)] = 0, \q\ae~t\in[0,1].
\end{equation}
Since $t\mapsto\dbE[u_1(t)]$ is square-integrable on $[0,1]$,  \rf{eqn:E-stationary*} implies
that $\dbE[X^*(1)]$ must be zero and hence $\dbE[u_1(t)]=0$ for almost every $t\in[0,1]$.
This yields a contradiction:
$$ 0 =\dbE[X^*(1)] = x + \int_0^1 \dbE[u_1^*(s)]ds = x. $$
Therefore, this problem has no open-loop saddle point for nonzero initial states.
\end{example}

In the previous discussion we have taken the starting time of the game to be zero for simplicity.
Sometimes it is convenient if we consider Problem (SLQG) over every subinterval $[t,T]$ of $[0,T]$.
In this case, the quadratic functional and the open-loop lower and upper values depend on the
initial time $t$ as well:
\begin{eqnarray*}
&\ds J(t,x;u_1,u_2) = \dbE\Bigg\{\lan GX(T),X(T)\ran +\int_t^T\Blan\!
   \begin{pmatrix}Q   & \!S_1^\top & \!S_2^\top \\
                  S_1 & \!R_{11}   & \!R_{12}   \\
                  S_2 & \!R_{21}   & \!R_{22}   \end{pmatrix}\!
   \begin{pmatrix}X \\ u_1 \\ u_2  \end{pmatrix}\!,
   \begin{pmatrix}X \\ u_1 \\ u_2  \end{pmatrix}\! \Bran ds\Bigg\}, \\
&\ds V^-(t,x) = \sup_{u_2\in\cU_2[t,T]} \inf_{u_1\in\cU_1[t,T]} J(t,x;u_1,u_2), \\
&\ds V^+(t,x) = \inf_{u_1\in\cU_1[t,T]} \sup_{u_2\in\cU_2[t,T]} J(t,x;u_1,u_2),
\end{eqnarray*}
where for $i=1,2$,
\begin{align*}
\cU_i[t,T] &=\ts\Big\{\f:[t,T]\times\Om\to\dbR^{m_i} \bigm|\f\in\dbF,~\dbE\int_t^T|\f(s)|^2ds<\i \Big\}.
\end{align*}
Obviously, with the initial time zero replaced by $t$, the previous results remain true.
Keeping this in mind, we now look at a special case of Problem (SLQG), the deterministic
two-person zero-sum LQ differential game, in which the diffusion part of the state equation
is absent, i.e.,
\begin{align}\label{case:D=0}
 C(s)=0, \q D_1(s)=0, \q D_2(s)=0, \q\forall s\in[0,T],
\end{align}

\begin{theorem}\label{thm:DLQG-main}
Let {\rm\ref{A1}--\ref{A2}} and \rf{case:D=0} hold. Suppose that
\begin{equation}\label{D=0:Rii>0}
   R_{11}\gg0, \q R_{22}\ll0.
\end{equation}
Then the following statements are equivalent:
\begin{enumerate}[\rm(i)]\setlength{\itemsep}{-1pt}
\item The open-loop lower and upper values $V^\pm(t,x)$ are finite for every initial pair $(t,x)$.
\item A unique open-loop saddle point exists for every initial pair $(t,x)$.
\end{enumerate}
Moreover, if the above statements hold true, then the Riccati equation
\begin{equation}\label{Ric:D=0}\left\{\begin{aligned}
 & \dot P + PA + A^\top P + Q -(PB+S^\top)R^{-1}(B^\top P+S)=0, \\
 & P(T)=G
\end{aligned}\right.\end{equation}
admits a unique solution $P\in C([0,T];\dbS^n)$, and the unique open-loop saddle point
$u^*=\begin{pmatrix}u_1^*\\u_2^*\end{pmatrix}$ for the initial pair $(t,x)$ is
given by the following closed-loop representation:
\begin{align}\label{DLQG:u*}
 u^*(s) = \Th(s)X^*(s), \q s\in[t,T],
\end{align}
where $\Th=-R^{-1}(B^\top P+S)$ and $X^*$ is the solution to the closed-loop system
$$\left\{\begin{aligned}
dX^*(s) &= [A(s)+B(s)\Th(s)]X^*(s)ds, \\
 X^*(t) &= x.
\end{aligned}\right.$$
\end{theorem}

In order to prove the above result, we need the following lemma.

\begin{lemma}\label{lmm:Yong}
Let $\Psi\in\dbR^{(2n)\times(2n)}$ be an invertible matrix and $\Sigma\in\dbR^{n\times n}$.
Suppose that for every $x\in\dbR^n$, there exists a $y\in\dbR^n$ such that
\begin{align}\label{Yong}
(\Si, -I_n)\Psi\begin{pmatrix}x \\ y\end{pmatrix}=0.
\end{align}
Then the $n\times n$ matrix $(\Si, -I_n)\Psi\begin{pmatrix}0 \\ I_n\end{pmatrix}$ is invertible.
\end{lemma}

\begin{proof}
Suppose to the contrary that there is a nonzero vector $\eta\in\dbR^n$ such that
\begin{align}\label{Yong:1}
 \eta^\top(\Si, -I_n)\Psi\begin{pmatrix}0 \\ I_n\end{pmatrix}=0.
\end{align}
Let $x\in\dbR^n$ be an arbitrary vector, and let $y=y(x)$ be such that \rf{Yong} holds. Then
$$ (\Si, -I_n)\Psi\begin{pmatrix}I_n\\0\end{pmatrix}x
= -(\Si, -I_n)\Psi\begin{pmatrix}0 \\ I_n\end{pmatrix}y, $$
and hence
$$ \eta^\top(\Si, -I_n)\Psi\begin{pmatrix}I_n\\0\end{pmatrix}x
= -\eta^\top(\Si, -I_n)\Psi\begin{pmatrix}0 \\ I_n\end{pmatrix}y =0. $$
Since $x\in\dbR^n$ is arbitrary, it follows from the above that
\begin{align}\label{Yong:2}
 \eta^\top(\Si, -I_n)\Psi\begin{pmatrix}I_n\\0\end{pmatrix}=0.
\end{align}
Combining \rf{Yong:1} and \rf{Yong:2} we obtain
\begin{align*}
 \eta^\top(\Si, -I_n)\Psi\begin{pmatrix}0&I_n\\I_n&0\end{pmatrix}=(0,0).
\end{align*}
Since $\Psi$ is invertible, we have $\eta^\top(\Si, -I_n)=(0,0)$ and hence $\eta^\top=0$.
This is a contradiction to the choice of $\eta$.
\end{proof}

\begin{proof}[\textbf{Proof of \autoref{thm:DLQG-main}}]
Clearly, (ii) implies (i). For the converse implication, we consider for $\l>0$,
the quadratic functional $J_{\l}(t,x;u_1,u_2)$ defined by
$$ J_{\l}(t,x;u_1,u_2) \deq J(t,x;u_1,u_2) + \l\dbE\int_t^T|u_1(s)|^2ds - \l\dbE\int_t^T|u_2(s)|^2ds. $$
Since $V^\pm(0,x)$ are finite for all $x$, we see from \autoref{thm:tu-ao} that
\begin{alignat*}{3}
J_\l(0,0;u_1,0) &\ges \l\|u_1\|^2,  \q&& \forall u_1\in\cU_1[0,T], \\
J_\l(0,0;0,u_2) &\les -\l\|u_2\|^2, \q&& \forall u_2\in\cU_2[0,T].
\end{alignat*}
Then it follows from \autoref{thm:main-Ric} that the following Riccati equation admits a solution
$P_\l\in C([0,T];\dbS^n)$:
$$\left\{\begin{aligned}
 & \dot P_\l + P_\l A + A^\top P_\l + Q -(P_\l B+S^\top)R_\l^{-1}(B^\top P_\l+S)=0, \\
 & P_\l(T)=G,
\end{aligned}\right.$$
where
$$ R_\l = \begin{pmatrix}R_{11}+\l I_{m_1} & R_{12} \\ R_{21} & R_{22}-\l I_{m_2}\end{pmatrix}. $$
Further, by \autoref{thm:main-saddle} and \autoref{prop:saddle-value},
$$ V_{\l}(t,x) = \lan P_\l(t)x,x\ran, \q\forall x\in\dbR^n. $$
Since $V^\pm(t,x)$ are finite for all $(t,x)$, we conclude by \autoref{prop:Vl-V} that
for each $t\in[0,T]$, $\{P_\l(t)\}$ has a convergent subsequence $\{P_{\l_k}(t)\}_{k=1}^\i$
($\lim_{k\to\i}\l_k=0$) with limit $P(t)$.
We claim that the function $P$ is a solution to the Riccati equation \rf{Ric:D=0}.
To this end, let us fix $t\in[0,T)$ and assume without loss of generality that $\{P_\l(t)\}$
itself converges to $P(t)$ as $\l\to0$.
Consider, for each $x\in\dbR^n$, the following matrix {\it forward} ODE:
\begin{equation*}\left\{\begin{aligned}
\dot X_\l(s) &= (A-BR_\l^{-1}S)X_\l - BR_\l^{-1}B^\top Y_\l,  \\
\dot Y_\l(s) &= -(A-BR_\l^{-1}S)^\top Y_\l - (Q-S^\top R_\l^{-1}S)X_\l,  \\
     X_\l(t) &= x, \q Y_\l(t)= P_\l(t)x.
\end{aligned}\right.\end{equation*}
It has a unique solution $(X_\l,Y_\l)\in C([t,T];\dbR^n)\times C([t,T];\dbR^n)$,
and one can verify directly that $X_\l$ is given by the ODE
\begin{equation*}\left\{\begin{aligned}
\dot X_\l(s) &= (A-BR_\l^{-1}S - BR_\l^{-1}B^\top P_\l)X_\l,  \\
     X_\l(t) &= x,
\end{aligned}\right.\end{equation*}
and that $Y_\l$ is given by
\begin{equation}\label{Yl=PlXl}
 Y_\l(s)=P_\l(s)X_\l(s),  \q s\in[t,T].
\end{equation}
Let $\Psi_\l(s)$ be the fundamental matrix for the homogeneous system
$$\dot\Bx(s)
= \begin{pmatrix}A-BR_\l^{-1}S & -BR_\l^{-1}B^\top     \\
        -(Q-S^\top R_\l^{-1}S) & -(A-BR_\l^{-1}S)^\top \end{pmatrix}\Bx(s), \q s\in[0,T]. $$
Then we have
$$\begin{pmatrix}X_\l(s) \\ Y_\l(s)\end{pmatrix}
= \Psi_\l(s)\Psi_\l(t)^{-1}\begin{pmatrix}x \\ P_\l(t)x\end{pmatrix}. $$
Note that as $\l\to0$, $P_\l(t)$ converges to $P(t)$, and for all $s\in[0,T]$,
$\Psi_\l(s)$ converges to $\Psi(s)$, the fundamental matrix for the homogeneous system
\begin{align}\label{Psi-Bx}
\dot\Bx(s) = \begin{pmatrix}A-BR^{-1}S & -BR^{-1}B^\top     \\
                   -(Q-S^\top R^{-1}S) & -(A-BR^{-1}S)^\top \end{pmatrix}\Bx(s), \q s\in[0,T].
\end{align}
Thus, $(X(s),Y(s))\deq\lim_{\l\to0}(X_\l(s),Y_\l(s))$ exists for every $s\in[t,T]$ and
$$\begin{pmatrix}X(s) \\ Y(s)\end{pmatrix}
= \Psi(s)\Psi(t)^{-1}\begin{pmatrix}x \\ P(t)x\end{pmatrix}. $$
Noting that by \rf{Yl=PlXl},
$$ Y(T)=\lim_{\l\to0}Y_\l(T)=\lim_{\l\to0}GX_\l(T)= GX(T), $$
we obtain
$$0 = (G, -I_n)\begin{pmatrix}X(T) \\ Y(T)\end{pmatrix}
    = (G, -I_n)\big[\Psi(T)\Psi(t)^{-1}\big]\begin{pmatrix}x \\ P(t)x\end{pmatrix}. $$
Since $x$ is arbitrary, it follows from \autoref{lmm:Yong} that
\begin{align}\label{Lambda-rep}
 \L(t)\deq (G, -I_n)\big[\Psi(T)\Psi(t)^{-1}\big]\begin{pmatrix}0 \\ I_n\end{pmatrix}
\end{align}
is invertible. Consequently, by noting that
\begin{align*}
& (G, -I_n)\big[\Psi(T)\Psi(t)^{-1}\big]\begin{pmatrix}0 \\ I_n\end{pmatrix}P(t)x
  +(G, -I_n)\big[\Psi(T)\Psi(t)^{-1}\big]\begin{pmatrix}I_n \\ 0\end{pmatrix}x \\
&\q= (G, -I_n)\big[\Psi(T)\Psi(t)^{-1}\big]\begin{pmatrix}x \\ P(t)x\end{pmatrix} \\
&\q= 0, \q\forall x\in\dbR^n,
\end{align*}
we obtain
\begin{align}\label{P-rep}
P(t) = -\L(t)^{-1}(G,-I_n)\big[\Psi(T)\Psi(t)^{-1}\big]\begin{pmatrix}I_n \\ 0\end{pmatrix}.
\end{align}
From \rf{P-rep} we see that the function $P$ is differentiable. By differentiating both sides of
$$ \L(t)P(t) = -(G,-I_n)\big[\Psi(T)\Psi(t)^{-1}\big]\begin{pmatrix}I_n \\ 0\end{pmatrix}$$
and then pre-multiplying by $\L(t)^{-1}$, it can be shown that $P$ satisfies the
Riccati equation \rf{Ric:D=0}. The uniqueness of a solution to \rf{Ric:D=0} can be
proved by a standard argument using Gronwall's inequality.
Having the existence of a solution to \rf{Ric:D=0}, we can use the same argument as
in the proof of \autoref{thm:main-saddle} to show that the pair $(u_1^*,u_2^*)$ defined
by \rf{DLQG:u*} is the unique open-loop saddle point for the initial pair $(t,x)$.
\end{proof}

From the above proof, we have the following corollary to \autoref{thm:DLQG-main}.

\begin{corollary}\label{crllry:rep-P}
Under the assumptions of \autoref{thm:DLQG-main}, the solution to the Riccati equation \rf{Ric:D=0}
admits the representation \rf{P-rep}, where $\Psi$ is the fundamental matrix for the homogeneous
system \rf{Psi-Bx} and $\L$ is given by \rf{Lambda-rep}.
\end{corollary}

\section{Conclusion}\label{Sec:Conclusion}

In this paper, we studied the open-loop saddle point, as well as the open-loop lower and upper values,
for two-person zero-sum stochastic LQ differential games with deterministic coefficients.
We derived a necessary condition \rf{J1>0} (respectively, \rf{J2<0}) for the finiteness of the open-loop
lower (respectively, upper) value (\autoref{thm:tu-ao}) and showed that under \ref{A3}, a condition
slightly stronger than \rf{J1>0}--\rf{J2<0}, an open-loop saddle point  uniquely exists and admits a
closed-loop representation (\autoref{thm:main-saddle}).
We found that the existence of an open-loop saddle point implies the finiteness of open-loop lower and
upper values (\autoref{prop:saddle-value}), but the latter does not even imply the existence of an
open-loop value (\autoref{V-<V<V+}).
The Riccati equation plays a crucial role throughout this paper.
By investigating the connection between the stochastic LQ differential game and two stochastic LQ
optimal control problems and examining the local existence of a solution to the Riccati equation,
we established the globally strongly regular solvability of the Riccati equation under the condition
\ref{A3} (\autoref{thm:main-Ric}).
We also presented an example showing that the strongly regular solvability of the Riccati equation
does not necessarily imply the condition \ref{A3} (\autoref{ex:PnotA3}).
\autoref{fig:summary} briefly summarizes these results.
For the deterministic two-person zero-sum LQ differential game, which can be regarded as a special
case of the stochastic game, we provided an alternative proof for Zhang's result \cite{Zhang2005}
on the equivalence of the existence of an open-loop saddle point and the finiteness of the open-loop
lower and upper values.
As a by-product of our approach, it was shown that the Riccati equation has an explicit solution when
the open-loop lower and upper values are finite (\autoref{thm:DLQG-main} and \autoref{crllry:rep-P}).

\ms

\begin{figure}[h]\centering\small
\begin{tikzpicture}[node distance=1.1cm, scale=0.6, transform shape]


\node (V1V2) [Block]{$V^-(x)>-\i$ and $V^+(x)<+\i$};

\node (C1) [Implication,below of=V1V2]{$\Big\Uparrow$};

\node (V) [Block,below of=C1]{$V(x)$ exists and is finite};

\node (C2) [Implication,below of=V]{$\Big\Uparrow$};

\node (Saddle) [Block,below of=C2]{Open-loop saddle point exists};

\node (C3) [Implication,below of=Saddle]{$\Big\Uparrow$};

\node (P1P2) [Block,below of=C3]{Both $P_1$ and $P_2$ exist};

\node (C4) [Implication,below of=P1P2]{$\Big\Downarrow$};

\node (P) [Block,below of=C4]{Riccati equation \rf{Ric} of the game has a solution $P$};


\node (V-) [Block,left of=V1V2,xshift=-6cm]{$V^-(x)>-\i$};

\node (L1) [Implication,below of=V-]{$\Big\Downarrow$};

\node (J1) [Block,left of=V,xshift=-6cm]{$J(0;u_1,0)\ges0$};

\node (L2) [Implication,below of=J1]{$\Big\Uparrow$};

\node (J1+) [Block,left of =Saddle,xshift=-6cm]{$\exists\a>0$~s.t.~$J(0;u_1,0)\ges\a\|u_1\|^2$};

\node (L3) [Implication,below of=J1+]{$\Big\Updownarrow$};

\node (P1) [Block,left of=P1P2,xshift=-6cm]{Riccati equation \rf{Ric:1} of Problem (SLQ)$_1$ has a solution $P_1$};


\node (V+) [Block,right of=V1V2,xshift=6cm]{$V^+(x)<+\infty$};

\node (R1) [Implication,below of=V+]{$\Big\Downarrow$};

\node (J2) [Block,right of=V,xshift=6cm]{$J(0;0,u_2)\les0$};

\node (R2) [Implication,below of=J2]{$\Big\Uparrow$};

\node (J2+) [Block,right of =Saddle,xshift=6cm]{$\exists\a>0$~s.t.~$J(0;0,u_2)\les-\a\|u_2\|^2$};

\node (R3) [Implication,below of=J2+]{$\Big\Updownarrow$};

\node (P2) [Block,right of=P1P2,xshift=6cm]{Riccati equation \rf{Ric:2} of Problem (SLQ)$_2$ has a solution $P_2$};


\draw [jian] (V1V2) to  (V-);
\draw [jian] (V1V2) to  (V+);
\draw [jian] (P1) to (P1P2);
\draw [jian] (P2) to (P1P2);

\end{tikzpicture}
\caption{Summary of the results}\label{fig:summary}
\end{figure}
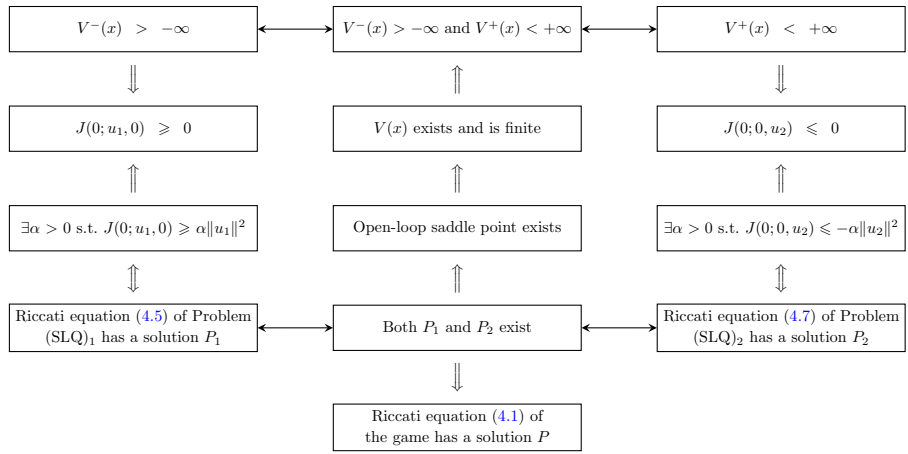

\bs

{\bf Acknowledgement.}
The author would like to thank Prof. Jiongmin Yong for his helpful advice on various technical issues
examined in this paper, which has led to this improved version of the paper.



\begin{thebibliography}{90}
\addtolength{\itemsep}{-1.0ex}

\bibitem{Basar-Bernhard1995}
\rm T.~Ba\c{s}ar and P.~Bernhard,
\it $H^\i$-Optimal Control and Related Minimax Design Problems: A Dynamic Game Approach,
\rm 2nd ed., Birkh\"{a}user Boston, Boston, 1995.


\bibitem{Bensoussan-Sung-Yam-Yung2016}
\rm A.~Bensoussan, K.~C.~J.~Sung, S.~C.~P.~Yam, and S.~P.~Yung,
\it Linear-quadratic mean field games,
\rm J. Optim. Theory Appl., 169 (2016), pp. 496--529.

\bibitem{Bernhard1979}
\rm P.~Bernhard,
\it Linear-quadratic, two-person, zero-sum differential games: Necessary and sufficient conditions,
\rm J. Optim. Theory Appl., 27 (1979), pp. 51--69.


\bibitem{Carmona2016}
\rm R.~Carmona,
\it Lectures on BSDEs, stochastic control, and stochastic differential games with financial applications,
\rm SIAM, 2016.


\bibitem{Delfour2007}
\rm M.~C.~Delfour,
\it Linear quadratic differential games: Saddle point and Riccati differential equations,
\rm SIAM J. Control Optim., 46 (2007), pp. 750--774.

\bibitem{Delfour-Sbarba2009}
\rm M.~C.~Delfour and O.~D.~Sbarba,
\it Linear quadratic differential games: Closed loop saddle points,
\rm SIAM J. Control Optim., 47 (2009), pp. 3138--3166.


\bibitem{Hamadene1998}
\rm S.~Hamad\`{e}ne,
\it Backward-forward SDE's and stochastic differential games,
\rm Stoch. Proc. Appl., 77 (1998), pp. 1--15.


\bibitem{Hamadene1999}
\rm S.~Hamad\`{e}ne,
\it Nonzero sum linear-quadratic stochastic differential games and backward-forward equations,
\rm Stoch. Anal. Appl., 17 (1999), pp. 117--130.


\bibitem{Engwerda2009}
\rm J.~Engwerda,
\it Linear quadratic differential games: an overview,
\rm in Advances in dynamic games and their applications, pp. 1--34. Birkh\"{a}user Boston, 2009.


\bibitem{Ho-Bryson-Baron1965}
\rm Y.~C.~Ho, A.~E.~Bryson, and S.~Baron,
\it Differential games and optimal pursuit-evasion strategies,
\rm IEEE Trans. Automat. Control, 10 (1965), pp. 385--389.


\bibitem{Mou-Yong2006}
\rm L.~Mou and J.~Yong,
\it Two-person zero-sum linear quadratic stochastic differential games by a Hilbert space method,
\rm J. Industrial Management Optim., 2 (2006), pp. 95--117.


\bibitem{Schmitendorf1970}
\rm W.~E.~Schmitendorf,
\it Existence of optimal open-loop strategies for a class of differential games,
\rm J. Optim. Theory Appl., 5 (1970), pp. 363--375.


\bibitem{Sun-Li-Yong2016}
\rm J.~Sun, X.~Li, and J.~Yong,
\it Open-loop and closed-loop solvabilities for stochastic linear quadratic optimal control problems,
\rm SIAM J. Control Optim., 54 (2016), pp. 2274--2308.


\bibitem{Sun-Xiong-Yong2020}
\rm J.~Sun, J.~Xiong, and J.~Yong,
\it Indefinite stochastic linear-quadratic optimal control problems with random coefficients:
    Closed-loop representation of open-loop optimal controls,
\rm arXiv:1809.00261v2.


\bibitem{Sun-Yong2014}
\rm J.~Sun and J.~Yong,
\it Linear quadratic stocahastic differential games: Open-loop and closed-loop saddle points,
\rm SIAM J. Control Optim., 52 (2014), pp. 4082--4121.


\bibitem{Sun-Yong2019}
\rm J.~Sun and J.~Yong,
\it Linear quadratic stochastic two-person nonzero-sum differential games: Open-loop and closed-loop
    Nash equilibria,
\rm Stoch. Proc. Appl., 129 (2019), pp. 381--418.


\bibitem{Sun-Yong2020}
\rm J.~Sun and J.~Yong,
\it Stochastic Linear-Quadratic Optimal Control Theory: Open-Loop and Closed-Loop Solutions,
\rm to appear in Springer Briefs in Mathematics, 2020.

\bibitem{Sun-Yong2020b}
\rm J.~Sun and J.~Yong,
\it Stochastic Linear-Quadratic Optimal Control Theory: Differential Games and Mean-Field Problems,
\rm to appear in Springer Briefs in Mathematics, 2020.

\bibitem{Sun-Yong-Zhang2016}
\rm J.~Sun, J.~Yong, and S.~Zhang,
\it Linear quadratic stochastic two-person zero-sum differential games in an infinite horizon,
\rm ESAIM Control Optim. Calc. Var., 22 (2016), pp. 743--769.


\bibitem{Ran-Yu-Zhang2020}
\rm T.~Ran, Z.~Yu, and R.~Zhang,
\it A closed-loop saddle point for zero-sum linear-quadratic stochastic differential games with mean-field type,
\rm Systems \& Control Letters, 136 (2020), 104624.




\bibitem{Yong2015}
\rm J.~Yong,
\it Differential Games --- A Concise Introduction,
\rm World Scientific Publisher, Singapore, 2015.




\bibitem{Yu2015}
\rm Z.~Yu,
\it An optimal feedback control-strategy pair for zero-sum linear-quadratic stochastic differential game:
    the Riccati equation approach,
\rm SIAM J. Control Optim., 53 (2015), pp. 2141--2167.


\bibitem{Zhang2005}
\rm P.~Zhang,
\it Some results on two-person zero-sum linear quadratic differential games,
\rm SIAM J. Control Optim., 43 (2005), pp. 2157--2165.


\end{thebibliography}
\end{document}